\documentclass[oneside,a4paper,10pt,reqno]{amsart}
\usepackage{amsfonts}
\usepackage{amssymb}
\usepackage{amsmath}
\usepackage[height=9in,a4paper,hmargin={2.5cm,2.5cm}]{geometry}
\usepackage{epsfig}
\usepackage{graphicx}
\usepackage{enumitem}
\usepackage{verbatim}
\usepackage[T1]{fontenc}
\usepackage[utf8]{inputenc}
\usepackage{mathtools}
\usepackage{esint}
\usepackage{graphicx}
\usepackage{mathrsfs}
\usepackage{multirow}
\usepackage{amsthm}
\usepackage[english]{babel}
\usepackage{amsfonts}
\usepackage{listings}
\usepackage{caption}
\usepackage{csquotes}
\usepackage{amsmath}
\usepackage{amssymb}
\usepackage{mathtools}
\usepackage{tabularx}
\usepackage{enumitem}
\usepackage{bm}
\usepackage{tikz}
\usepackage{pgf,tikz, tikz-3dplot} 
\usetikzlibrary{arrows.meta}
\usetikzlibrary{intersections}

\usepackage{pgfplots} 
\numberwithin{equation}{section}
\pgfplotsset{/pgf/number format/use comma,compat=newest}
\theoremstyle{plain}
\newtheorem{thm}{Theorem}[section]

\newtheorem{lem}[thm]{Lemma}
\newtheorem{prop}[thm]{Proposition}
\theoremstyle{definition}
\newtheorem{defn}[thm]{Definition}
\newtheorem{rem}[thm]{Remark}

\newcommand{\R}{\mathbb{R}}

\newcommand{\N}{\mathbb{N}}
\newcommand{\Z}{\mathbb{Z}}

\newcommand{\Sf}{\mathbb{S}}

\newcommand{\dist}{\textnormal{dist}}
\newcommand{\sign}{\textnormal{sign}}

\newcommand{\diam}{\textnormal{diam}}

\newcommand{\sbv}{SBV}

\usepackage{pgf,tikz, tikz-3dplot} 
\usetikzlibrary{arrows.meta}
\usetikzlibrary{intersections}
\usepackage{floatrow}
\usepackage{enumerate}
\usepackage{enumitem}
\usepackage{hyperref}
\usepackage[normalem]{ulem}

\definecolor{blue_links}{RGB}{13,0,180} 

\usepackage{hyperref}
\hypersetup{
	colorlinks=true, 
	linktoc=all,     
	linkcolor=blue_links,  
	citecolor=blue_links,
	urlcolor=blue_links,
}

\usepackage{mathtools}

\usepackage{color}

\newcommand\scircle[4]{%
  \tdplotsetrotatedcoords{#2}{#3}{0}                                            
  \let\a\tdplotalpha                                                            
  \let\b\tdplotbeta                                                             
  \let\p\tdplotmainphi                                                          
  \let\t\tdplotmaintheta                                                        
  \pgfmathsetmacro\azx{cos(\a)*cos(\b)*sin(\p)*sin(\t) - sin(\b)*cos(\t) - cos(\b)*cos(\p)*sin(\a)*sin(\t)}
  \pgfmathsetmacro\azy{-cos(\a)*cos(\p)*sin(\t) - sin(\a)*sin(\p)*sin(\t)}
  \pgfmathsetmacro\azz{cos(\b)*cos(\t) + cos(\a)*sin(\b)*sin(\p)*sin(\t) - cos(\p)*sin(\a)*sin(\b)*sin(\t)}
  \pgfmathsetmacro\re {#1*cos(#4)}                                              
  \pgfmathsetmacro\ze {#1*sin(#4)}                                              
  \pgfmathsetmacro\coX{\ze*cos(#2)*sin(#3)}                                     
  \pgfmathsetmacro\coY{\ze*sin(#2)*sin(#3)}                                     
  \pgfmathsetmacro\coZ{\ze*cos(#3)}                                             
  \coordinate (coffs) at (\coX,\coY,\coZ);                                      
  \tdplotsetrotatedcoordsorigin{(coffs)}                                        
  \begin{scope}[tdplot_rotated_coords]                                          
    \pgfmathsetmacro\tanEps{tan(#4)}                                            
    \pgfmathsetmacro\bOneside{((\tanEps)^2)>=(((\azx)^2+(\azy)^2)/(\azz)^2)}    
    \ifthenelse{\bOneside=1}{
      \pgfmathsetmacro\bFrontside{(\azx*\re+\azz*\ze)>=0}                       
       \ifthenelse{\bFrontside=1}                                               
         {\draw (0,0) circle (\re);}                                            
         {\draw[dashed] (0,0) circle (\re);}                                    
    }{
      \pgfmathsetmacro\u{\azy}                                                  
      \pgfmathsetmacro\v{sqrt( (\azx)^2 + (\azy)^2 - (\azz)^2*(\tanEps)^2 )}    
      \pgfmathsetmacro\w{\azx - \azz*\tanEps}                                   
      \pgfmathsetmacro\phiBf{2*atan2(\u-\v,\w)}                                 
      \pgfmathsetmacro\phiFb{2*atan2(\u+\v,\w)}                                 
      \pgfmathsetmacro\bUnwrapA{(\phiFb-\phiBf)>360}                            
      \pgfmathsetmacro\bUnwrapB{\phiBf>\phiFb}                                  
      \ifthenelse{\bUnwrapA=1}{\pgfmathsetmacro\phiBf{\phiBf+360}}{}            
      \ifthenelse{\bUnwrapB=1}{\pgfmathsetmacro\phiBf{\phiBf-360}}{}            
      \draw[dashed] (\phiFb:\re) arc (\phiFb:{\phiBf+360}:\re);                 
      \draw (\phiBf:\re) arc (\phiBf:\phiFb:\re);                               
    }                                                                           
  \end{scope}                                                                   
}

\title[From discrete to continuum in the helical XY model: emergence of chirality transitions in the $S^1$ to $S^2$ limit]{From discrete to continuum in the helical XY-model: emergence of chirality transitions in the $S^1$ to $S^2$ limit}

\begin{document}

\author{Marco Cicalese}
\address[Marco Cicalese]{Department of Mathematics, Technische Universit\"at M\"unchen, Boltzmannstrasse 3, 85748 Garching, Germany}
\email{cicalese@ma.tum.de}

\author{Dario Reggiani}
\address[Dario Reggiani]{Applied Mathematics M\"unster, University of M\"unster, Einsteinstrasse 62, 48149 M\"unster, Germany}
\email{dario.reggiani@uni-muenster.de}

\author{Francesco Solombrino}
\address[Francesco Solombrino]{Universit\'a del Salento, via provinciale per Arnesano, 73100 Lecce, Italy}
\email{francesco.solombrino@unina.it}

	\maketitle
	\begin{abstract}We analyze the discrete-to-continuum limit of a frustrated ferromagnetic/anti-ferromagnetic $\Sf^2$-valued spin system on the lattice $\lambda_n\Z^2$ as $\lambda_n\to 0$. For $\Sf^2$ spin systems close to the Landau-Lifschitz point (where the helimagnetic/ferromagnetic transition occurs), it is well established that for chirality transitions emerge with vanishing energy. Inspired by recent work on the $N$-clock model, we consider a spin model where spins are constrained to $k_n$ copies of $\Sf^1$ covering $\Sf^2$ as $n\to\infty$. We identify a critical energy-scaling regime and a threshold for the divergence rate of $k_n\to+\infty$, below which the $\Gamma$-limit of the discrete energies capture chirality transitions while retaining an $\Sf^2$-valued energy description in the continuum limit.
	\end{abstract}
	
\tableofcontents

\section{Introduction}

Spatial variations in magnetization often arise due to the interplay of competing magnetic interactions, geometric constraints, and lattice symmetries, which can prevent the system from simultaneously satisfying all local interaction preferences. These mechanisms, collectively termed frustration, play a crucial role in enabling polarization to be induced and controlled. For its importance in the emergence of ferroelectricity in multiferroic materials geometric frustration has garnered significant attention from both the physical and mathematical communities over the past decades (see \cite{diep} for a comprehensive review). A paradigmatic class of systems exhibiting magnetic frustration is provided by classical spin models with competing ferromagnetic and antiferromagnetic interactions, which can lead to helical ground states. Despite extensive investigations, the rigorous understanding of the complete phase diagram of such systems remains an open challenge. In this paper we focus on the helical XY spin model on the square lattice $\Z^2$, a prototypical example of frustrated system that we analyze, after appropriately scaling its interaction parameters near the Landau-Lifschitz point, where helical ground states are expected to emerge. 

\bigskip

The energy of the helical XY spin model on the square lattice $\mathbb{Z}^2$ in the configuration $u:i\in\mathbb{Z}^2\mapsto u^{i}\in \Sf^2$ is
\begin{equation}\label{HXY}
E(u)=-\sum_{i\in\Z^{2}}\left(J_0 (u^{i},u^{i+e_1})-J_1(u^{i},u^{i+2e_1})+J_2(u^{i},u^{i+e_2})\right),
\end{equation}
where $(\cdot,\cdot)$ denotes the scalar product in $\R^3$ and the constants $J_0$ and $J_{1}$ are the interaction parameters for the nearest-neighbors (NN) and the next-to-nearest-neighbors (NNN) interactions in the horizontal direction $e_{1}$, respectively, while $J_{2}$ is the interaction parameter for the NN interactions in the vertical direction $e_{2}$.  
Under the assumption $J_0, J_{1}, J_{2}>0$, the system exhibits distinct behaviors along the two lattice directions. In the $e_{2}$-direction, the interaction potential $-J_2(u^{i},u^{i+e_2})$ promotes a ferromagnetic spin order with all spins aligned. Conversely, in the $e_{1}$-direction, competing effects arise: ferromagnetic (F) nearest-neighbor (NN) interactions, described by the potential $-J_0 (u^{i},u^{i+e_1})$, favor alignment, whereas antiferromagnetic (AF) next-nearest-neighbor (NNN) interactions, with potential $J_1(u^{i},u^{i+2e_1})$, favor anti-alignment, that is neighboring spins of opposite orientations. This competition induces a frustration mechanism, evident even along a single horizontal lattice line. The latter is the energy associated with interactions in the $e_{1}$-direction and is given by:
\begin{equation}\label{intro:energy-1d}
F(u)=-\sum_{j\in\Z}\left(J_0 (u^{j},u^{j+1})-J_1(u^{j},u^{j+2})\right).
\end{equation}
The energy above models a so-called F/AF {\it frustrated chain} for which it is well known that no spin configuration can minimize all interactions simultaneously (see \cite{diep} for a comprehensive analysis of frustrated spin systems). In order to simplify the geometry of the ground states of the helical $XY$ model we scale the interaction parameters by letting $J_2 \to +\infty$, hence enforcing spin alignment in the $e_{2}$-direction. As a result, spin configurations $u$ with the energy in \eqref{HXY} finite have a one-dimensional profile; i.e., $u^{(i_{1},i_{2})}=v^{i_{1}}$ for some $v:\Z\to \Sf^{2}$ and their geometry can be understood from the geometry of spin fields $v$ with finite one-dimensional energy $F(v)$. 
The analysis of the F/AF frustrated chain model on $\Z$ in the vicinity of the helimagnetic/ferromagnetic transition point and in the so-called discrete-to-continuum limit was first conjectured in \cite{DmiKriext} (an extended version of \cite{DmiKri}) and subsequently rigorously carried out via $\Gamma$-convergence. In this direction initial results were obtained for spins constrained to $\Sf^{1}$ in \cite{CiSol} (see also \cite{SciVall} and see \cite{CiFoOr_1,CiFoOr_2} for two possible extensions of this $\Sf^1$ model to $\Z^2$) and later extended to the $\Sf^{2}$ case in \cite{CiRuSol} (see also \cite{KuLa}). Before explaining in details our main results we briefly summarize those obtained in the papers \cite{CiSol} and \cite{CiRuSol}.\\

In \cite{CiSol} the continuum limit of the F/AF chain energy in \eqref{intro:energy-1d} has been studied in the case of $\Sf^{1}$-valued spins. After scaling the functional by a small parameter $\frac{\lambda_n}{J_{1}}$ (${\lambda_n}\to0$ as $n\to +\infty$), and setting $\Z_{n}=\{j \in \Z:\  \lambda_{n} j \in [0,1]\}$ one defines $F_{n}:\{u:j\in\Z_{n}\mapsto u^{j}\in S^{1}\}\to\R$ as 
\begin{equation}\label{intro:energy-n}
F_{n}(u)=-\alpha\sum_{j\in\Z_{n}}{\lambda_n}(u^{j},u^{j+1})+\sum_{j\in\Z_{n}}{\lambda_n}(u^{j},u^{j+2})
\end{equation}
where $\alpha=J_{0}/J_{1}$ is known as the frustration parameter and completely characterizes the ground states of $F_{n}$. More precisely, neighboring spins are aligned if $\alpha\geq 4$ (ferromagnetic phase), while they form a constant angle $\varphi=\pm\arccos(\alpha/4)$ if $0<\alpha<4$ (helimagnetic phase). When the system is in the helimagnetic phase it shows a chiral symmetry. This means that two different families of ground states are possible according to the two possible choices of $\varphi$, corresponding to either clockwise or counter-clockwise handedness of the helical configuration (or analogously spin angular variation) or to what is known as chirality $+1$ or $-1$. When $\alpha$ is close to $4$ the system is at the Landau-Lifschitz point, that is at the onset of frustration. At this point the energy necessary to break the chiral symmetry can be found letting the frustration parameter $\alpha$ approach $4$ from below. In \cite{CiSol} the authors consider this limit within the discrete-to-continuum analysis of the model. They let $\alpha$ depend on $n$ and replace in \eqref{intro:energy-n} $\alpha$ by $\alpha_{n}=4(1-{\delta_n})$ for some vanishing sequence ${\delta_n}>0$. They then refer the energy to its minimum by introducing the functional $H_{n}(u):=F_{n}(u)-\min F_{n}$ and consider the asymptotic behavior as the lattice spacing $\lambda_n\to 0$ of a properly scaled version of $H_n$.  In \cite{CiSol} the minimal energy necessary to brake the chirality symmetry in the continuum limit is found by computing the $\Gamma$-limit of $H_{n}/({\lambda_n}{\delta_n}^{3/2})$ with respect to the $L^{1}$ convergence of the chirality order parameter (a properly scaled version of the angular increment between neighboring spins) as ${\lambda_n}\to 0$. More precisely for ${\lambda_n}/\sqrt{{\delta_n}}\to 0$ (at other scalings chirality transitions are either forbidden or not penalized)  it is proved that states $u_n$ with $H_{n}(u_n)/({\lambda_n}{\delta_n}^{3/2})$ uniformly bounded, have chiralities that converge in the continuum limit to piece-wise constant functions. The energy cost of a chirality transition is then obtained and the $\Gamma$-limit of $H_n/({\lambda_n}{\delta_n}^{3/2})$ computed, resulting proportional to the number of jumps of the chirality, namely the number of times the spin configuration changes the sign of its angular velocity. Since the limiting energy penalizes phase transitions, the result in \cite{CiSol} falls under the category of discrete-to-continuum problems discussed in the monograph \cite{ABCS}. In the case of $\Sf^{2}$-valued spins the picture drastically changes. This time the chirality is a vectorial order parameter entailing not only the handedness of the helical configuration but also the axis of rotation which can now vary.  While one can still renormalize the energy and show that the modulus of the chirality remains constant for ground states, the transition mechanism between ground states of differing chiralities is different. In \cite{CiRuSol} it has been proved that the ($\Sf^{2}$) helical XY system can adjust its chirality continuously over a slow spatial scale paying little energy in the continuum limit. This phenomenon is rooted in the behavior of the renormalized energy, which at leading order resembles a discrete vectorial Modica-Mortola-type functional in the chirality variable. The potential term associated with this functional has connected wells, enabling low energy transitions. This result aligns with the behavior of Modica-Mortola functionals in the continuum, as rigorously pointed out in the works \cite{Ambrosio} and \cite{Baldo}.

\medskip
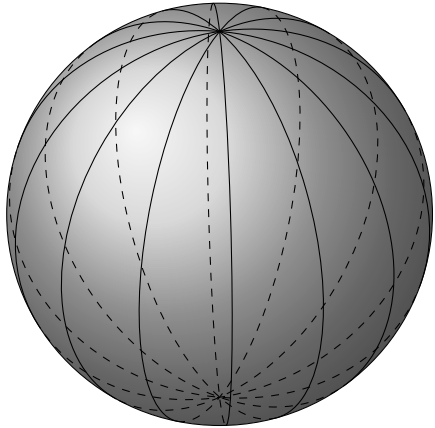
\begin{figure}[H]
\tdplotsetmaincoords{60}{125}    
\begin{tikzpicture}[scale=1.4]

\draw[tdplot_screen_coords, ball color = black!20, opacity=0.4] (0,0,0) circle (2); 

\foreach \angle in {0,51.4,...,308.57} {
    \scircle{2}{\angle}{90}{0};
}

\end{tikzpicture}
\caption{$\Sf^2$ and an example of $k=7$ admissible copies of $\Sf^1$}
\label{intro:fig}
\end{figure}
It is worth noting that, in the continuum framework, a similar interplay between ferromagnetic and antiferromagnetic (or alignment and anti-alignment) behavior arises in phase transition problems. This occurs when the energy includes two (or more) terms depending on derivatives of the phase variable of different orders, contributing to the total energy with opposing signs (see, for instance, \cite{FM, CDMFL, CSZ, FHLZ}). 

In this paper, we aim to better understand the role of the topology of the spin target space in the emergence of chirality transitions in helical systems. Our work is inspired by the pioneering paper \cite{FS} and \cite{COR1, COR2, COR3}, where a question was posed in the case of the $XY$ (rotator) model, comparing it to the so-called $N$-clock model. Here, we extend this comparison to the $\Sf^2$ helical $XY$ model by introducing a multi-dimensional analog of the $N$-clock model. We constrain the spins to belong to a finite collection of $k$ distinct copies of $\Sf^1$, obtained by rotating a maximal circle in $\Sf^2$ by integer multiples of $2\pi/k$ around a fixed diameter (see Figure \ref{intro:fig}). We investigate the behavior of these spin systems in the continuum limit as before close to the ferromagnetic/helimagnetic transition point (for $\alpha_n=(4-\delta_n)$, $\delta_n\to 0$ and $\lambda_n/\sqrt{\delta_n}\to 0$) when $k=k_n\to\infty$. To this end we first introduce the family of energies $H^{k_n}_n$ defined such that $H^{k_n}_n(u)=H_n(u)$ if the spin field $u$ respects the constraints while $H^{k_n}_n(u)=+\infty$ otherwise. We then compute via $\Gamma$-convergence (with respect to the $L^{1}$ convergence of a properly scaled chirality order parameter associated to $u$ as in the $\Sf^2$ case) the discrete-to-continuum limit of the family of energies $H^{k_n}_n/(\lambda_n \delta_n^{3/2})$. As the lattice spacing $\lambda_n\to 0$, roughly speaking the collection of $k_n$ copies of $\Sf^1$ gradually covers $\Sf^2$. If $k_n$ grows sufficiently slowly, the spins cannot modify their chirality continuously over a slow spatial scale at zero energy cost, unlike in the unconstrained $\Sf^2$ case.
In Theorem \ref{thm: Gamma lim < 1 quarto}, assuming $1 \ll k_n \ll \delta_n^{-1/4}$, we show that the admissible chiralities are piece-wise constant $BV$ functions, that each chirality transition pays a fixed energy cost, and the limiting energy is proportional to the number of transitions, hence mirroring the behavior of the $\Sf^1$ case (see \cite{CiRuSol} for analogous results on helical $XY$ models where $\Sf^2$ spins are constrained to finitely many $\Sf^1$ copies). Furthermore, in Theorem \ref{thm: critical case}, we establish that the upper bound $\delta_n^{-1/4}$ is sharp. Specifically, for $k_n \sim \delta_n^{-1/4}$, the limit chirality is still a function of bounded variation, but it is no longer piecewise constant. Moreover the limiting energy scales as its total variation and when a chirality transition occurs, both the $\Gamma\text{-}\liminf$ and $\Gamma\text{-}\limsup$ of $\frac{H^{k_n}_n}{\lambda_n \delta_n^{3/2}}$ may depend on the traces of the chirality. A detailed analysis of this dependence lies beyond the scope of the present work. From a technical perspective, the variational analysis of the present model amounts to analyze the functional $H_n^{k_n}/{\lambda_n \delta_n^{3/2}}$ which, in terms of a properly scaled chirality variable can be written as a discrete approximation of a vector-valued Modica-Mortola-type functional (we refer to \cite{BraYip} for the scalar case). In our framework, the considered functionals involve $k_n$ disconnected wells that, as $n\to +\infty$ converge to a connected set (see also \cite{CriGra, CriFonGan, BraZepp} for the study of phase field functionals with varying wells in the context of homogenization). In this respect our proof strategy highlights how the topology of the spin target space in our helical $XY$ model, explicitly examined through the introduction of a higher-dimensional analog of the $N$-clock model, can be related to the study of phase-transition functionals whose potential wells undergo topological transformations as the transition length approaches zero. 


	\section{Notation and Preliminaries}
	
	For $d, k \in \N$, we denote by $\mathcal{L}^d$ the Lebesgue measure in $\R^d$ and by $\mathcal{H}^k$ the $k$-dimensional Hausdorff measure in $\R^d$. The Lebesgue measure of a set $E \subset \R^d$ is indicated with $|E|$. For every $r > 0$ and every $x \in \R^d$, we denote by $B_r(x)$ the open
	ball in $\R^d$ of radius $r$ and center $x$. If $x=0$ we only write $B_r$. The symbol $\Sf^N$ stands for the unit sphere in $\R^{N+1}$. Given two vector $a,b \in \R^m$ we denote with $(a,b)$ their scalar product. Given two sequences of real numbers $a_n,b_n$ we write $a_n \sim b_n$ if there exist two positive constants $c_1$ and $c_2$ such that $c_1 b_n \leq a_n \leq c_2 b_n$ for every $n \geq 1$.
	
We now briefly recall some basic definitions about bounded variation functions. We refer to \cite[Section 4.5]{Ambrosio2000FunctionsOB} for more details on this topic. Let $\Omega \subset \R^d$ be an open set. The space $\textnormal{BV}(\Omega,\R^m)$ of functions of bounded variation is the set of $u \in L^1(\Omega;\R^m)$ whose distributional gradient $Du$ is a bounded Radon measure on $\Omega$ with values in $\R^{m \times d}$. Given $u \in \textnormal{BV}(\Omega,\R^m)$ we can write $Du=D^a u+ D^s u$, where $D^a u$ is absolutely continuous and $D^s u$ is singular with respect to $\mathcal{L}^d$. The set $J_{u}$ denotes the jump set of~$u$, $\nu_{u}$ stands for the approximate unit normal to~$u$, and $[u]:= u^{+} - u^{-}$ represents the jump of~$u$, that is, the difference between the traces~$u^{+}$ and~$u^{-}$ of~$u$ on~$J_{u}$, defined according to the orientation of~$\nu_{u}$. The set $J_u$ is countably rectifiable, while the density $\nabla u \in L^1(\Omega,\R^{m \times d})$ of $D^a u$ with respect to $\mathcal{L}^d$ coincides a.e. in $\Omega$ with the approximate gradient of $u$. The space $\sbv(\Omega,\R^m)$ of special functions of bounded variation is defined as the set of all $u \in \textnormal{BV}(\Omega,\R^m)$ such that $|D^s u|(\Omega \setminus J_u)=0$. Moreover, we denote by $\sbv_{\textnormal{loc}}(\Omega,\R^m)$ the space of functions belonging to $\sbv(U,\R^m)$ for every $U \Subset \Omega$. Given a set $\Omega \subset \R^n$ we say that the countable family $\{E_\alpha\}_\alpha$ is a Caccioppoli partition of $\Omega$ if $\mathcal{L}^d(\Omega \setminus \cup_\alpha E_\alpha)=0$ and $\sum_\alpha P(E_\alpha;\Omega)=0$, where $P(E;\Omega)$ is the perimeter of the set $E$ relative to $\Omega$.
	
We now present the definition of bounded variation piecewise constant functions on $\Omega$.
	
	\begin{defn}\label{sbvpc}
		Given $u \in \sbv(\Omega;\R^m)$ we say that $u \in  BV_{\textnormal{pc}}(\Omega;\R^m)$ if there exists a Caccioppoli partition of $\Omega$ such that $u$ is constant on each element of the partition.
	\end{defn}
Let $\Omega$ be a bounded closed subset of $\R^n$ with Lipschitz boundary, such that $0 \in \Omega$, and let $\{\lambda_n\}_n$ be a vanishing sequence. We call $\Z_n^d(\Omega)$ the set of all $i \in \Z^d$ such that $\lambda_n i \in \Omega$, that is
	\begin{equation}\label{def Zn}
		\Z_n^d(\Omega):=\{i \in \Z^d \colon \ \lambda_n i \in \Omega \}=\Z^d \cap \lambda_n^{-1}\Omega.
	\end{equation}
	We introduce the triangulation $\mathcal{T}_n$ of $\Omega$ associated to $\lambda_n$ as the triangulation obtained by Freudenthal (also called Kuhn) decomposition of the cubes with vertices $\lambda_n i$ for every $i \in \Z_n^d(\Omega+B_{\sqrt{d}\lambda_n})$.
	We also define the space $\mathcal{U}_n^d(\Omega;\R^m)$ as the space of functions $z \colon i \in \Z_n^d(\Omega) \to \R^m$. Given $z \in \mathcal{U}_n^d(\Omega;\R^m)$ we will often use $z^i$ to indicate $z(i)$.
	
We also need the notion of affine interpolant.
	
	\begin{defn}\label{affine interpolant}
		Given $z \in \mathcal{U}_n^d(\Omega,\R^m)$ we call the affine interpolant of $z$, and we indicate it with $\tilde{z}$, the piecewise affine function $\tilde{z} \colon \Omega \to \R^m$ such that $\tilde{z}(\lambda_n i)=z(i)$ for every $i \in \Z_n^d(\Omega)$ and which is affine on the triangulation $\mathcal{T}_n$ of $\Omega$.
	\end{defn}
	
	Finally, we recall the notion of Kuratowski limit.
	
	\begin{defn}\label{def: kur}
		Given a sequence of closed sets $\{E_n\}_n$ and a closed set $E$ in $\R^N$, we say that the sequence $E_n$ converges to $E$ in the Kuratowski sense if, given a sequence of points $\{x_n\}_n$ with $x_n \in E_n$ for every $n \geq 1$, we have that any cluster point $x$ of $\{x_n\}_n$ belongs to $E$, and every $x \in E$ is the limit of a sequence $\{x_n\}_n$ with $x_n \in E_n$. Alternatively, if and only if the sequence of functions $\dist(\cdot,E_n)$ converges locally uniformly to $\dist(\cdot,E)$ on $\R^N$. 
	\end{defn}
In what follows we introduce the spin model we are interested in and some of the results already obtained in previous papers.\\ 

A configuration for the helical XY $\Sf^2$-spin model on the square lattice $\Z^2$ is a map $u \colon \Z^2 \to \Sf^2$ whose energy is given by 
	\begin{equation}\label{energy}
		E(u)=-\sum_{i \in \Z^2} \left( J_0 (u^i,u^{i+e_1})-J_1(u^i,u^{i+2e_1})+J_2(u^i,u^{i+e_2}) \right),
	\end{equation}
	where the constants $J_0$ and $J_1$ are the interaction parameters for the nearest-neighbors (NN) and the next-to-nearest-neighbors (NNN) interactions in the horizontal direction $e_1$, respectively, while $J_2$ is the interaction parameter for the NN interactions in the vertical direction.
	We let the three parameters $J_0,J_1,J_2>0$ be scale independent. 
	
	Let $\Omega:=[0,1]^2$. Recalling the definition of $\Z_n^2(\Omega)$ in \eqref{def Zn}, we define $\mathcal{U}^2_n(\Omega;\Sf^2)$ as the space of functions $u \in \Z^2_n(\Omega) \mapsto \Sf^2$ and $\overline{\mathcal{U}}^2(\Omega;\Sf^2)$ as the subspace of those functions $u$ such that, for all $m \in \Z \cap [0,[1/\lambda_n]],$ it holds
	\begin{equation}\label{U overline}
		\left( u^{(1,m)},u^{(0,m)} \right)=\left( u^{([1/\lambda_n],m)},u^{([1/\lambda_n]-1,m)} \right),
	\end{equation}
	where $[\cdot]$ denotes the integer part. We moreover set $R^1_n(\Omega):=\{ i \in \Z_n^2(\Omega) \colon \ i+2e_1 \in \Z_n^2(\Omega)\}$ and $R^2_n(\Omega):=\{ i \in \Z_n^2(\Omega) \colon \ i+e_2 \in \Z_n^2(\Omega)\}$. \\
	
	Without loss of generality we scale the energy \eqref{energy} by $J_{1,n}:=J_1/\lambda_n$ and rename $J_{0,n}$ and $J_{2,n}$ accordingly, obtaining the energy for  $u \in \overline{\mathcal{U}}^2(\Omega;\Sf^2)$ given by
	\begin{equation}\label{energyn}
	E_n(u):= - \sum_{i \in R^1_n(\Omega)} \lambda^2_n {J_{0,n}(u^i,u^{i+e_1})-(u^i,u^{i+2e_1})}-J_{2,n} \sum_{i \in R^2_n(\Omega)} \lambda_n^2 (u^i,u^{i+e_2}).
	\end{equation}
	We refer to \cite{CiSol, CiRuSol} for the computation of the energy $E_n$ ground states. Taking $J_{0,n}>4$ leads to trivial ferromagnetic ground states. Hence, we focus on the asymptotics of the renormalized energy 
	$$
	H_n(u):= \frac{1}{2} \left( \sum_{i \in R^1_n(\Omega)} \lambda_n^2 \left| u^i-\frac{J_{0,n}}{2} u^{i+e_1}+u^{i+2e_1}  \right|^2+J_{2,n} \sum_{i \in R^2_n(\Omega)} \lambda_n^2\left|u^{i+e_2}-u^i\right|^2 \right),
	$$
	when the parameter $J_{0,n}$ is in the vicinity of the Landau-Lifschitz point $J_0=4$ and the parameter $J_{2,n}$ diverges. To this aim, we introduce $\delta_n \to 0$ and consider $J_{0,n}=4(1-\delta_n)$. We thus rewrite the energy $H_n$ as
	\begin{equation}\label{energy Hn}
		H_n(u)= \frac{1}{2} \left( \sum_{i \in R^1_n(\Omega)} \lambda_n^2 \left|u^i-2(1-\delta_n)u^{i+e_1}+u^{i+2e_1}\right|^2+J_{2,n} \sum_{i \in R^2_n(\Omega)} \lambda_n^2\left|u^{i+e_2}-u^i\right|^2 \right).
	\end{equation}
	Within this choice stable states have a one-dimensional helical structure and may
	exhibit chirality transitions in the propagation direction (the horizontal axis in our case). Consequently, the analysis we are going to perform starts by considering energies on one-dimensional horizontal slices of the domain. To this aim we introduce some additional notation. Let $I=[0,1]$ and define $\Z_n(I)$ as the set of points $i \in \Z$ such that $\lambda_n i \in [0,1]$. Define also $R_n(I):=\{ i \in \Z_n(I) \colon \ i+2 \in \Z_n(I) \}$. Similarly as in the two-dimensional setting we denote with $\mathcal{U}_n(I;\Sf^2)$ the space of functions $u \colon i \in \Z_n(I) \mapsto u^i \in \Sf^2$ and with $\overline{\mathcal{U}}_n(I;\Sf^2)$ the subspace of those $u$ such that
	\begin{equation}
		\left( u^1,u^0 \right)=\left( u^{[1/\lambda_n]},u^{[1/\lambda_n]-1} \right).
	\end{equation}
	It is convenient also to consider the following class of piecewise constant functions: given $u \in \overline{\mathcal{U}}_n(I)$ we associate to it a piecewise constant interpolation belonging to the class
	$$
	C_n(I;\Sf^2):=\{ u \in \overline{\mathcal{U}}_n(I;\Sf^2) \colon \ u(x)=u^i \ \ \mbox{if} \ x \in \lambda_n(i+[0,1)), \ i \in \Z_n(I) \}.
	$$
	The one-dimensional (sliced) renormalized energy is denoted by $H_n^{sl} \colon L^\infty(I;\R^3) \to [0,+\infty]$ and it is given by
	\begin{equation}\label{Hn energy sliced}
		H_n^{sl}(u):= 
		\begin{cases*}
			\displaystyle \frac{1}{2} \sum_{i \in R_n(\Omega)} \lambda_n \left|u^i-2(1-\delta_n)u^{i+1}+u^{i+2}\right|^2 & if $u \in C_n(I,\Sf^2)$, \\
			+\infty & otherwise.
		\end{cases*}
	\end{equation}
	It was proven in \cite[Proposition 2.1]{CiRuSol} that the zero order $\Gamma$-limit of $H_n^{sl}$ is trivial. 
	\begin{prop}
		Let $H_n^{sl} \colon L^\infty (I;\R^3) \to [0,+\infty]$ be the functional defined in \eqref{Hn energy sliced}. The $\Gamma \mbox{-} \lim_n H_n^{sl}$ with respect to the weak*-convergence in $L^\infty$ is given by
		\begin{equation*}
			H^{sl}(u)=
			\begin{cases*}
				0 & if $|u| \leq 1$, \\
				+\infty & otherwise.
			\end{cases*}
		\end{equation*}
	\end{prop}
	\noindent This degeneracy of minima for $H^{sl}$ suggests to perform a higher order analysis by $\Gamma$-convergence. In order to study the higher order asymptotic behavior of the one dimensional renormalized energy we start by introducing an order parameter. Given a function $u \in C_n(I;\Sf^2)$, for all $i \in \{ 1,\dots, [1/\lambda_n]-1\}$, we set
	\begin{equation}\label{theta}
		\theta^i(u) = \arccos((u^i,u^{i+1})) \in [0,\pi],
	\end{equation}
	and $w^i = u^i \times u^{i+1}$. We also introduce a new order parameter $z \colon \Z_n(I) \to \R^3$ defined as
	\begin{equation}\label{zeta}
		z^i=\frac{u^i \times u^{i+1}}{\sqrt{2\delta_n}}=\frac{w^i}{\sqrt{2\delta_n}},
	\end{equation}
	which represents a rescaled angular velocity (to be well defined we set $z^{[1/\lambda_n]}:=z^{[1/\lambda_n]-1}$). Such $z$ will be extended in $L^1(I;\R^3)$ by means of a piecewise constant interpolation. We define the map $T_n \colon C_n(I;\Sf^2) \to L^1(I;\R^3)$ as the map that associate to each $u$ the corresponding $z$ according to \eqref{zeta}. Notice that the map $T_n$ is not injective and if $u$ satisfies periodic boundary conditions in the sense of \eqref{U overline}, then $|z|$ is periodic and vice-versa. Thus, we define the energy $H^{sl}_n$ on $L^1(I;\R^3)$ as
	\begin{equation}\label{no??}
		H^{sl}_n(z):=
		\begin{cases*}
			\displaystyle \inf_{T_n(u)=z} H^{sl}_n(u)  & if $z=T_n(u)$ for some $u \in C_n(I;\Sf^2)$, \\
			+\infty & otherwise.
		\end{cases*}
	\end{equation}
	Taking the infimum in the definition above has no effect in the asymptotic analysis we are going to perform (see \cite[Remark 3.1]{CiRuSol}). However, without restriction on the possible values the configuration $u$ can take on $\Sf^2$, it was proven in \cite{CiRuSol} that the $\Gamma$-limit of $H^{sl}_n$ is trivial as well. That is, in contrast to the $\Sf^1$-valued spin system, for $\Sf^2$-valued spins the functional $H^{sl}_n$ does not penalize chirality transitions between ground states and the optimal asymptotic energy for a transition turns out to be zero. The following result is contained in \cite[Theorem 3.1]{CiRuSol}.
	\begin{thm}
		Let $H_n^{sl} \colon L^1(I;\R^3) \to [0,+\infty]$ be defined as in \eqref{no??}. Assume that $\frac{\lambda_n}{\sqrt{\delta_n}} \to 0$. Then the rescaled functionals $\frac{H^{sl}_n}{\sqrt{2}\lambda_n\delta_n^{3/2}}$ $\Gamma$-converge with respect to the weak*-convergence in $L^\infty(I)$ to the functional
		\begin{equation*}
			H^{sl}(z)=
			\begin{cases*}
				0 & if $z \in L^\infty(I,B_1)$, \\
				+\infty & otherwise.
			\end{cases*}
		\end{equation*}
	\end{thm}
		\section{Main results}
In order to better understand the role of topology in the emergence of chirality transitions, drawing inspiration from \cite{FS} and \cite{COR1, COR2, COR3}, we modify the energy $H^{sl}_n$ introduced in \eqref{no??} by adding a constraint on the possible values the spin variable $u$ can take on $\Sf^2$. Namely, we impose that $u$ can take values only on a subset of $\Sf^2$ consisting of $k$ copies of $\Sf^1$, where $k \in \N$. We show that if $k=k_n$ diverge "slowly enough" as $n \to +\infty$, we can detect a non-trivial limit energy. To carry out the analysis we need to fix additional notation. Given $k \in \N$, let $q_1,\dots,q_k$ be a family of $k$ distinct points in $\Sf^2$. For every $l \in \{1,\dots,k\}$ we set $S^1_l:=\Sf^2 \cap q_l^\perp$, where $q_l^\perp$ is the orthogonal complement of $q_l$. We also define 
	$$
	Q_k:=\{q_1,\dots,q_k\}, \qquad M_k:=\bigcup_{l=1}^k S^1_l, \qquad \mathfrak{L}_k:=\bigcup_{l=1}^k \textnormal{span}(q_l).
	$$
	We restrict the spin variable $u$ to take values only in $M_k$. Namely, we consider the subset $C_n(I;M_k)$ of $C_n(I;\Sf^2)$, defined as $C_n(I;\Sf^2)$ but with $M_k$ in place of $\Sf^2$, and we define the energy $H_n^{sl,k} \colon L^1(I) \to [0,+\infty]$ as
	\begin{equation}\label{Hslkn}
		H^{sl,k}_n(z):=
		\begin{cases*}
			\displaystyle \inf_{T_n(u)=z} H^{sl}_n(u)  & if $z=T_n(u)$ for some $u \in C_n(I;M_k)$, \\
			+\infty & otherwise.
		\end{cases*}
	\end{equation}
	Along the sequence the value $k$ and the $k$ points $\{q_1,\dots,q_k\}$ can depend on $n$. Thus, to highlight the possible dependence of $k$ on $n$, we will write $k_n$ in place of $k$. Under the assumption $k_n \delta_n^{1/4} \to 0$ we show that we can recover compactness and $\Gamma$-convergence for the functionals $H^{sl,k_n}_n$ to a non-trivial limit functional.
	
	\begin{prop}\label{prop: z compact < 1 quarto}
		Assume that $\lambda_n/\sqrt{\delta_n} \to 0$ and consider a sequence $k_n$ in $\N$ such that $k_n\delta_n^{1/4} \to 0$. Let $z_n=T_n(u_n)$ for some $u_n \in C_n(I,M_{k_n})$ be such that
		\begin{equation}\label{0-1}
			H^{sl,k_n}_n(z_n) \leq C \lambda_n \delta_n^{3/2}, \qquad\mbox{for some} \ \ C>0.
		\end{equation} 
	Then, denoting by $\overline{Q}$ the Kuratowski limit of the sets $\{Q_{k_n}\}_n$, there exists $z \in BV_{\textnormal{pc}}(I,\overline{Q})$ such that (up to subsequences) $z_n \to z$ strongly in $L^1(I)$.
	\end{prop}

	\begin{thm}\label{thm: Gamma lim < 1 quarto}
	Let $H^{sl,k_n}_n \colon L^1(I) \to [0,+\infty]$ be defined as in \eqref{Hslkn}. Assume that $\lambda_n/\sqrt{\delta_n} \to 0$ and $k_n \delta_n^{1/4} \to 0$. Then, the rescaled functionals $\frac{H^{sl,k_n}_n}{\sqrt{2}\lambda_n \delta_n^{3/2}}$ $\Gamma$-converge with respect to the strong $L^1$-topology to the functional
	\begin{equation*}
		\begin{cases*}
				\displaystyle \frac{8}{3} \# S(z) & if $z \in BV_{\textnormal{pc}}(I,\overline{Q})$; \\
				+\infty & otherwise.
		\end{cases*}
	\end{equation*}
	\end{thm}

	The proof of Proposition \ref{prop: z compact < 1 quarto} and Theorem \ref{thm: Gamma lim < 1 quarto} is based on some general results about discrete phase field functionals with co-domain constraints which we present in Section \ref{susect}.
	
	Thanks to Proposition \ref{prop: z compact < 1 quarto} and Theorem \ref{thm: Gamma lim < 1 quarto} we can study the asymptotic behavior of the renormalized XY-model, that is, the energy $H_n$ defined in \eqref{energy Hn} scaled by $\lambda_n \delta_n^{3/2}$, in the limit of strong ferromagnetic interactions. To be precise, we assume that
	\begin{equation}\label{assumptions}
	\frac{\lambda_n}{\sqrt{\delta_n}} \to 0, \qquad J_{2,n} \frac{\lambda_n}{\sqrt{\delta_n}} \geq c >0.
	\end{equation}
	We embed the energies in a common function space, to this end we identify every function $u \in \overline{\mathcal{U}}^2(\Omega;\Sf^2)$ with its piecewise constant interpolation belonging to the space
	$$
	C_{n,2}(\Omega;\Sf^2):=\left\{ u \in \overline{\mathcal{U}}^2(\Omega;\Sf^2) \colon \ u(x)=u(\lambda_n i) \mbox{ if $x \in \lambda_n(i+[0,1)^2), i \in \Z^2_n(\Omega)$}\right\}.
	$$
	We can extend the functional $H_n$ to a functional defined on $L^\infty(\Omega;\Sf^2)$ setting
	\begin{equation*}
		H_n(u):=
		\begin{cases*}
			H_n(u) & if $u \in C_{n,2}(\Omega;\Sf^2)$, \\
			+\infty & otherwise.
		\end{cases*}
	\end{equation*}
	Given $u \in C_{n,2}(\Omega;\Sf^2)$ we define $z \in C_{n,2}(\Omega;\R^3)$ as
	$$
	z^i := \frac{u^i \times u^{i+e_1}}{\sqrt{2\delta_n}}
	$$
	and write for short $z:=T_n(u)$. Given a family of $k$ points $Q_k \subset \Sf^2$, we can define the constrained energy $H_n^k \colon L^1(\Omega;\R^3) \to [0,+\infty]$ setting
	\begin{equation}\label{Hkn}
		H^{k}_n(z):=
		\begin{cases*}
			\displaystyle \inf_{T_n(u)=z} H_n(u)  & if $z=T_n(u)$ for some $u \in C_{n,2}(\Omega;M_k)$, \\
			+\infty & otherwise.
		\end{cases*}
	\end{equation}
	The following result holds.
	\begin{thm}
		Let $H^k_n \colon L^1(\Omega;\R^3) \to [0,+\infty]$ be defined as in \eqref{Hkn}, assume that $k_n \delta_n^{1/4} \to 0$ and \eqref{assumptions} hold. Then, the rescaled functionals $\frac{H_n^k}{\sqrt{2}\lambda_n \delta_n^{\frac{3}{2}}}$ $\Gamma$-converge with respect to the strong $L^1$-topology to the functional
		\begin{equation*}
			\begin{cases*}
				\displaystyle \frac{8}{3} \mathcal{H}^1(J_z) & if $z \in BV_{\textnormal{pc}}(\Omega,\overline{Q})$ and $z$ does not depend on $y$; \\
				+\infty & otherwise.
			\end{cases*}
		\end{equation*}
		Above, $\overline{Q}$ is defined as in Proposition \ref{prop: z compact < 1 quarto}.
	\end{thm}

	\begin{proof}
		The proof of this result follows by Proposition \ref{prop: z compact < 1 quarto} and Theorem \ref{thm: Gamma lim < 1 quarto} arguing exactly as in Section 4 of \cite{CiRuSol}. Hence, we omit it.
	\end{proof}

	Finally, if $k_n \sim \delta_n^{-1/4}$ we show that the previous considerations do not hold in general. Indeed, we have the following result.
	
	\begin{thm}\label{thm: critical case}
	Assume that $\frac{\lambda_n}{\delta_n^{3/4}} \to 0$ as $n \to +\infty$ and let $k_n=2\lceil \delta_n^{-1/4} \rceil$. There exist configurations $Q_{k_n}$ and two constants $C_1,C_2>0$ such that 
	$$
	C_1 |Dz|(I) \leq \Gamma \mbox{-} \liminf_{n \to +\infty} \frac{H^{sl,k_n}_n}{\sqrt{2}\lambda_n \delta_n^{3/2}}(z) \leq  \Gamma \mbox{-}\limsup_{n \to +\infty} \frac{H^{sl,k_n}_n}{\sqrt{2}\lambda_n \delta_n^{3/2}}(z) \leq C_2 |Dz|(I).
	$$
	for every $z \in BV(I,\overline{Q})$. 
	\end{thm}

	\begin{rem}
		Theorem \ref{thm: critical case} implies that in the critical case $k_n \sim \delta_n^{-1/4}$, there exists configurations such that the $\Gamma$-limit behaves asymptotically as a total variation.
		In particular, in the aforementioned case, we have that for a general $z \in BV_{\textnormal{pc}}(I,\overline{Q})$
		$$
		\Gamma \mbox{-} \liminf_{n \to +\infty} \frac{H_n^{sl,k_n}}{\sqrt{2}\lambda_n\delta_n^{3/2}}(z) \neq \frac{8}{3}S(z) \qquad \mbox{and} \qquad \Gamma \mbox{-}\limsup_{n \to +\infty} \frac{H_n^{sl,k_n}}{\sqrt{2}\lambda_n\delta_n^{3/2}}(z) \neq \frac{8}{3}S(z).
		$$
		Moreover, both the $\Gamma \mbox{-}\liminf$ and the $\Gamma \mbox{-}\limsup$ may in general depend on the traces $z^+$ and $z^-$ of a function $z \in BV_{\textnormal{pc}}(I,\overline{Q})$ on jump points. The analysis of such a dependence requires additional discussions which would exceed the scope of this paper.
	\end{rem}

	\section{Discrete phase transitions with co-domain constraints}\label{susect}
	
	This section contains the results regarding phase field functionals defined on a discretized domain and with co-domain constraints. First, we fix some notation.
	
	Let $|\cdot|_A \colon \R^N \to \R$ be a norm on $\R^N$ such that there exist $c_1,c_2>0$ with the property that
	\begin{equation}\label{A norm}
	c_1 |\xi| \leq |\xi|_A \leq c_2 |\xi| \qquad \mbox{for every $\xi \in \R^N$},
	\end{equation}
	and the set $\{\xi \in \R^N \colon \ |\xi|_A=1\}=:K$ is a compact subset of $\R^N$.
	Let $\varepsilon_n$ and $r_n$ be two vanishing positive sequences, and let $k_n \in \N$ a possibly diverging sequence. For every $n \in \N$, we consider $k_n$ points $\{q_j^n\}_{j=1}^{k_n} \subset K$. Let $Q_n:=\{q_j^n\}_{j=1}^{k_n}$ and $\overline{Q} \subseteq K$ the Kuratowski limit of the sets $Q_n$ as $n \to +\infty$. We set $\mathfrak{L}^n_j:=\textnormal{span}(q_n^j)+B_{r_n}$, for $j=1,\dots,k_n$ and $\mathfrak{L}^n:=\cup_{j=1}^{k_n} \mathfrak{L}_j^n$. We make the following assumptions:
	\begin{itemize}
		\item[(i)] $\frac{\lambda_n }{\varepsilon_n r_n^2}$ is a bounded sequence (the cylinders $\mathfrak{L}_j^n$ must not be too close to one another);
		\item[(ii)] for every $\eta>0$, there exists $\overline{n}=\overline{n}(\eta)$ such that for every $n \geq \overline{n}$ we have
		$$
		\mathfrak{L}_j^n \cap \mathfrak{L}_m^n \subset B_\eta, \qquad \mbox{for every $j,m=1,\dots,k_n$ with $j \neq m$.}
		$$
	\end{itemize}

	Set $g_n \colon \R^N \to \R^+ \cup \{+\infty\}$ as
	\begin{equation}\label{gn}
		g_n(\xi):=
		\begin{cases*}
			(|\xi|_A^2-1)^2 & $\xi \in \mathfrak{L}^n$; \\
			+\infty & otherwise.
		\end{cases*}
	\end{equation}

	We set $R_n(\Omega) \subset \Z^d$ as the set of all $i \in \Z_n^d(\Omega)$ such that $i+e_\ell \in \Z_n^d(\Omega)$ for every $\ell=1,\dots,d$. We are interested in the asymptotic analysis of the following energies for $z_n \in \mathcal{U}^d_n(\Omega;\R^{N})$
	\begin{equation}\label{01}
		\mathcal{F}_n(z_n):=\frac{1}{\varepsilon_n} \sum_{i \in R_n(\Omega)} \lambda_n^d g_n(z_n^i)+\varepsilon_n \sum_{i \in R_n(\Omega)} \lambda_n^d \left( \sum_{\ell=1}^d \left| \frac{z_n^i-z_n^{i+ e_\ell}}{\lambda_n} \right|^2 \right).
	\end{equation}

	The following compactness result holds true.

	\begin{prop}\label{prop: cpt MM}
		Let $z_n \in \mathcal{U}^d_n(\Omega;\R^N)$ be a sequence such that for every $n \geq 1$
		\begin{equation}\label{MM bound}
			\mathcal{F}_n(z_n) \leq C \qquad \mbox{for some $C>0$}.
		\end{equation}
		Then, there exists $z \in SBV_{\textnormal{pc}}(\Omega;\overline{Q})$ such that, up to subsequences, $z_n \to z$ in $L^1(\Omega;\R^N)$.
	\end{prop}

	Thanks to Proposition \ref{prop: cpt MM}, under the additional requirement that $\lim_n \frac{\lambda_n}{\varepsilon_n r_n^2} \to 0$, we can prove the following lower bound for the $\Gamma$-limit. We define the function $h \colon \overline{Q} \times \overline{Q} \to \R^+$ as
	\begin{equation}\label{G lim func}
		h(q_1,q_2):=\frac{4}{3}(|q_1|+|q_2|).
	\end{equation}

	\begin{thm}\label{G liminf}
		Let $z_n \in \mathcal{U}_n^d(\Omega;\R^N)$ be a sequence such that $z_n \to z$ in $L^1(\Omega;\R^N)$. Then, if $\frac{\lambda_n}{\varepsilon_n r_n^2} \to 0$,
		\begin{equation}\label{G liminf ex}
			\liminf_{n \to +\infty} \mathcal{F}_n(z_n) \geq \int_{J_z} h(z^-,z^+) \, d \mathcal{H}^{d-1}.
		\end{equation}
	\end{thm}

	\begin{rem}
		Notice that the domain of definition of the function $h$ is a subset of $K$, that is the $1$-level set of the norm $|\cdot |_A$. Moreover, indicating with $\hat{q}$ the unit vector $q/|q|$, the transition energy $h$ can also be rewritten as
		$$
		h(q_1,q_2)=\frac{4}{3} \left( \frac{1}{|\hat{q}_1|_A}+\frac{1}{|\hat{q}_2|_A} \right).
		$$
	\end{rem}

	\begin{rem}
		The condition $\lim_n \frac{\lambda_n}{\varepsilon_n r_n^2} \to 0$ ensures that the transitions between two different $\mathfrak{L}^n_j$'s can be detected in the limit by the norm of $z_n$ being close to zero.
	\end{rem}

	Finally, an upper bound matching with the lower bound holds. This entails a $\Gamma$-convergence result when $\varepsilon_n \to 0$ for the energies defined in \eqref{01}. Notice that the assumptions on $\lambda_n,\varepsilon_n$ and $r_n$ are weaker for the upper bound to hold.
	
	\begin{thm}\label{limsup}
		Let $z \in BV_{\textnormal{pc}}(\Omega;\overline{Q})$. Then, there exists a sequence $z_n \in \mathcal{U}_n^d(\Omega;\R^N)$ such that $z_n \to z$ in $L^1(\Omega;\R^N)$ and
		$$
		\limsup_{n \to +\infty} \mathcal{F}_n(z_n) \leq \int_{J_z} h(z^-,z^+) \, d \mathcal{H}^{d-1},
		$$
		where $h$ is the function in \eqref{G lim func}.
	\end{thm}

	We will use the following compactness result for piecewise constant functions defined on Caccioppoli partitions (see \cite{Ambrosio2000FunctionsOB}).
	
	\begin{thm}\label{thm: caccio cpt}
		Let $\Omega \subset \R^d$ be an open bounded set with Lipschitz boundary. Let $\{u_n\}_n$ be a sequence in $\sbv_{\textnormal{pc}}(\Omega;\R^N)$ such that $\sup_n \left( \Vert u_n \Vert_{L^\infty(\Omega)}+\mathcal{H}^{d-1}(J_{u_n}) \right) <+\infty$. Then, there exits $u \in \sbv_{\textnormal{pc}}(\Omega;\R^N)$ and a subsequence (not relabeled) such that $u_n \to u$ in $L^1(\Omega;\R^N)$.
	\end{thm}

	\begin{proof}[Proof of Proposition \ref{prop: cpt MM}]
		Observe that by \eqref{MM bound} we have $z_n \in \mathfrak{L}^n$ almost everywhere on $\Omega$, moreover, the functions $g_n$ are independent of $n$ on the set $\mathfrak{L}^n$, continuous and bounded below by $c|\xi|^4$ for some $c>0$ when $|\xi|\gg 1$. Hence, using standard arguments, we infer that the sequence $\{z_n\}$ is bounded in $L^1(\Omega;\R^N)$ and equiintegrable on $\Omega$. Consider the sequence 
		\begin{equation*}
			T(z_n):=
			\begin{cases*}
				\displaystyle  2\frac{z_n}{|z_n|_A} & if $|z_n|_A \geq 2$, \\
				z_n & otherwise.
			\end{cases*}
		\end{equation*}
		By \eqref{A norm} and the definition of $g_n$ we have $\mathcal{F}_n(T(z_n)) \leq \mathcal{F}_n(z_n)$. We claim that if $T(z_n)$ converges in $L^1$ to $z \in BV(\Omega;\overline{Q})$, then $z_n$ converges to $z$ in $L^1$ on $\Omega$ as well. Indeed, since $\sup_n \mathcal{F}_n(z_n)<+\infty$ by \eqref{MM bound}, we have that $|z_n|_A \to 1$ pointwise on $\Omega$ as $n \to +\infty$. Hence, $|\{|z_n|_A \geq 2\}| \to 0$ as $n \to +\infty$. Recalling that $|T(z_n)|_A \leq 2$ by construction and the fact that the sequence $z_n$ is equiintegrable on $\Omega$, we conclude that $z_n \to z$ in $L^1$ as well. Thus, since it is enough to prove compactness for equibounded sequences on $\Omega$, we may assume that $ \sup_n \Vert z_n \Vert_{L^\infty(\Omega;\R^N)} \leq C$. Let us define the piecewise constant scalar function $v_n \in \mathcal{U}_n^d(\Omega)$ as $v_n:=|z_n|_A$. By \eqref{gn}, \eqref{MM bound}, triangular inequality and recalling \eqref{A norm}, we have
		\begin{equation}\label{02}
			\sup_n \left[ \frac{1}{\varepsilon_n} \sum_{i \in R_n(\Omega)} \lambda_n^d ((v_n^i)^2-1)^2+\varepsilon_n \sum_{i \in R_n(\Omega)} \lambda_n^d \left( \sum_{\ell=1}^d \left| \frac{v_n^i-v_n^{i+ e_\ell}}{\lambda_n} \right|^2 \right)\right] \leq (1+c_2^2)C.
		\end{equation}
		Consider piecewise affine function $\tilde{v}_n \in H^1(\Omega)$ associated to $v_n$ according to Definition \ref{affine interpolant}. Given $\sigma>0$ small, using the equi-boundedness of $z_n$ on $\Omega$ and \eqref{02} we estimate
		\begin{align*}
			\int_\Omega ((\tilde{v}_n)^2-1)^2 \, dx & \leq (1+3\sigma) \sum_{i \in R_n(\Omega)} \lambda_n^d ((v_n^i)^2-1)^2 + \left(1+\frac{1}{\sigma}\right) \int_\Omega \left( (v_n)^2-(\tilde{v}_n)^2 \right)^2 \, dx \\
			& \leq (1+3\sigma) \sum_{i \in R_n(\Omega)} \lambda_n^d ((v_n^i)^2-1)^2 + C \left(1+\frac{1}{\sigma}\right)  \int_\Omega \left( v_n-\tilde{v}_n \right)^2 \, dx  \\
			& \leq (1+3\sigma) \sum_{i \in R_n(\Omega)} \lambda_n^d ((v_n^i)^2-1)^2 + C \left(1+\frac{1}{\sigma}\right) \sum_{i \in R_n(\Omega)} \lambda_n^d \left( \sum_{\ell=1}^d \left| v_n^i-v_n^{i+ e_\ell} \right|^2 \right) \\
			&  \leq (1+3\sigma) C \varepsilon_n + C(1+c_2^2) \left(1+\frac{1}{\sigma}\right) \frac{\lambda_n^2}{\varepsilon_n}.
		\end{align*}
		Hence, fixing $\sigma>0$, we have that
		\begin{align}
			\begin{split}\label{03}
			\sup_n & \left( \frac{1}{\varepsilon_n} \int_\Omega (\tilde{v}_n^2-1)^2 \, dx + \varepsilon_n \int_\Omega |\nabla \tilde{v}_n|^2 \, dx \right) \\
			& = \sup_n \left[ \frac{1}{\varepsilon_n} \int_\Omega (\tilde{v}_n^2-1)^2 \, dx +\varepsilon_n \sum_{i \in R_n(\Omega)} \lambda_n^d \left( \sum_{\ell=1}^d \left| \frac{v_n^i-v_n^{i+ e_\ell}}{\lambda_n} \right|^2 \right)\right] \\
			& \leq C+\frac{\lambda_n^2}{\varepsilon_n^2}C \leq C(1+r_n^2) \leq C.
			\end{split}
		\end{align}
		Using the coarea formula we infer
		$$
		C \geq \int_\Omega 2(\tilde{v}_n^2-1)|\nabla \tilde{v}_n| \, dx \geq \frac{3}{2}\int_{1/4}^{1/2}  \mathcal{H}^{d-1}(\partial^* \{ \tilde{v}_n <s \} \setminus \partial \Omega) \, ds.
		$$
		Therefore, by the mean value theorem, for every $n \geq 1$ there exists $t_n \in (1/4,1/2)$ such that 
		\begin{equation}\label{04}
			\mathcal{H}^{d-1}(\partial^* \{ \tilde{v}_n <t_n \} \setminus \partial \Omega) \leq C.
		\end{equation}
		Moreover,
		\begin{equation}\label{05}
			|\{ \tilde{v}_n <t_n \} \cap \Omega| \leq \frac{16}{9} \int_\Omega (\tilde{v}_n^2-1)^2 \, dx \leq \varepsilon_n C \to 0.
		\end{equation}
		We define the set of indexes where the difference between two adjacent points is bigger than $r_n/4$, namely
		\begin{equation*}
			\mathcal{I}_n:=\left\{ i \in R_n(\Omega) \colon \ \max_{\ell=1,\dots,d} \left\{\left|z_n^i-z_n^{i\pm e_\ell}\right|\right\} \geq \frac{r_n}{4} \right\}.
		\end{equation*}
		Notice that if we define
		$$
		\mathcal{J}_n :=\left\{ i \in R_n(\Omega) \colon \ \max_{\ell=1,\dots,d} \left\{\left|z_n^i-z_n^{i+  e_\ell}\right|\right\} \geq \frac{r_n}{4} \right\},
		$$
		then we have that $\# \mathcal{J}_n \geq \frac{\# \mathcal{I}_n}{\# ([-1,1]^d \cap \Z^d)}$.
		Hence, in view of \eqref{MM bound} we can estimate
		\begin{align*}
			C & \geq \varepsilon_n \sum_{i \in R_n(\Omega)} \lambda_n^d \left( \sum_{\ell=1}^d \left| \frac{z_n^i-z_n^{i+ e_\ell}}{\lambda_n} \right|^2 \right) \geq \sum_{i \in \mathcal{J}_n} \lambda_n^{d-2} \varepsilon_n \left( \sum_{\ell=1}^d \left| z_n^i-z_n^{i+ e_\ell} \right|^2 \right) \geq \lambda_n^{d-2} \varepsilon_n \left(\# \mathcal{J}_n \right) \frac{r_n^2}{16}.
		\end{align*}
		This gives that 
		\begin{equation}\label{06}
			\# \mathcal{I}_n \leq \frac{C}{\lambda_n^{d-2}\varepsilon_n r_n^2}.
		\end{equation}
		Correspondingly, we consider the set $\mathcal{Q}_n:=\bigcup_{i \in \mathcal{I}_n} \left( \lambda_n i+[0,\lambda_n)^n \right)$ and observe that by \eqref{06} 
		\begin{equation}\label{07}
			\mathcal{H}^{d-1}(\partial^* \mathcal{Q}_n \setminus \partial \Omega) \leq 2d \lambda_n^{d-1} \frac{C}{\lambda_n^{d-2}\varepsilon_n r_n^2}=C\frac{\lambda_n}{\varepsilon_n r_n^2} \leq C, \qquad |\mathcal{Q}_n| \leq \lambda_n^d \frac{C}{\lambda_n^{d-2}\varepsilon_n r_n^2} \leq C\lambda_n \frac{\lambda_n}{\varepsilon_n r_n^2} \leq \lambda_n C.
		\end{equation}
		We define the set $K_n:=(\mathcal{Q}_n \cup \{ \tilde{v}_n < t_n\}) \cap \Omega$. Let $E \subset \Omega$ be a connected component of $\Omega \setminus K_n$. We claim that there exists $j \in \{1,\dots,k_n\}$ such that for every $i \in \Z^d_n(\Omega)$ with $\lambda_n i \in \lambda_n \Z^d \cap E$ we have $z^i_n \in \mathfrak{L}^n_j$ when $n \geq 1$ is large enough. Indeed, since $\tilde{v}_n \geq t_n \geq 1/4$ on $E$, by definition of the piecewise affine interpolation and since $E \cap \mathcal{Q}_n=\emptyset$, we also have on $E$ that 
		$$
		|z_n| \geq \frac{1}{c_2} |z_n|_A = \frac{1}{c_2}v_n \geq \frac{1}{4c_2}-\frac{r_n}{4} \geq \frac{1}{8c_2}.
		$$
		Moreover, in view of assumption (ii), if we take $n \geq 1$ large enough we have that $\mathfrak{L}^n_j \cap \mathfrak{L}_m^n \subset B_{1/{16c_2}}$ for every $j,m=1,\dots,k_n$ with $j \neq m$. Hence, by definition of $\mathfrak{L}_j^n$, we deduce that 
		$$
		\dist\left(\mathfrak{L}_j^n \setminus B_{1/8c_2},\mathfrak{L}^n_m \setminus B_{1/8c_2}\right) \geq r_n, \qquad \mbox{for every $j,m = 1,\dots,k_n$ with $j \neq m$}.
		$$
		Fix $i \in \Z^d_n(\Omega)$ such that $\lambda_n i \in \lambda_n \Z^d \cap E$. By \eqref{gn} and \eqref{MM bound} there exists a fixed index $j \in \{1,\dots,k_n\}$ such that $z^i_n \in \mathfrak{L}^n_j$. In turn, recalling that $i \notin \mathcal{I}_n$, we have that $z_n^{i \pm  e_\ell} \in \mathfrak{L}_j^n$ for every $\ell=1,\dots,d$ as well. Therefore, since $E \cap \mathcal{Q}_n =\emptyset$, we can repeat the same argument for every $i \in \Z^d_n(\Omega)$ such that $\lambda_n i \in \lambda_n \Z^d \cap E$ and we conclude the claim. 
		
		Let $\{E^n_\alpha\}_\alpha \cup K_n$ be the partition of $\Omega$ induced by $\Omega \setminus K_n$, where $E_\alpha^n$ is a connected component of $\Omega \setminus K_n$. For every $\alpha$ there exists $j=j(\alpha) \in \{1,\dots,k_n\}$ such that $z_n \in \mathfrak{L}^n_j$ and $|z_n| \geq 1/4$ restricted to $E_\alpha^n$. Let us define the piecewise constant function $\hat{z}_n$ on $\Omega$ as
		\begin{equation}\label{07,5}
			\hat{z}_n:=
			\begin{cases*}
				q_n^{j(\alpha)} & on $E_\alpha^n$; \\
				0 & on $K_n$.
			\end{cases*}
		\end{equation}
		We estimate
		\begin{equation}\label{08}
			\Vert z_n - \hat{z}_n \Vert_{L^1(\Omega)} \leq \Vert z_n \Vert_{L^1(K_n)}+ \bigg \Vert z_n-\frac{z_n}{|z_n|_A} \bigg \Vert_{L^1(\Omega \setminus K_n)}+|\Omega| \bigg \Vert \frac{z_n}{|z_n|_A}-\hat{z}_n \bigg \Vert_{L^\infty(\Omega \setminus K_n)}.
		\end{equation}
		Using \eqref{05}, \eqref{07} and the fact that the sequence $z_n$ is equibounded, we infer that the first right-hand-side term of \eqref{08} vanishes in the limit. The second term also goes to zero since the sequence $z_n$ is equibounded and $|z_n|_A \to 1$ in measure on $\Omega$. Finally, the third term vanishes by definition of $\mathfrak{L}^n_j$.
		Recalling \eqref{04} and \eqref{07}, we infer that $\hat{z}_n \in BV_{\textnormal{pc}}(\Omega;\R^N)$ is such that $\Vert \hat{z}_n \Vert_{BV(\Omega;\R^N)} \leq C$ for every $n \geq 1$ large enough. Therefore, using Theorem \ref{thm: caccio cpt} we infer that there exists $z \in BV_{\textnormal{pc}}(\Omega;K)$ such that u.t.s. $\hat{z}_n \to z$ in $L^1(\Omega;\R^N)$. In particular, by definition of the sequence $\hat{z}_n$ we must have $z(x) \in \overline{Q}$ for almost every $x \in \Omega$. Thus, we conclude that $z_n \to z$ in $L^1(\Omega;\R^N)$ and $z \in BV_{\textnormal{pc}}(\Omega;\overline{Q})$.
	\end{proof}

	We now prove the liminf inequality using a slicing argument.

	\begin{proof}[Proof of Theorem \ref{G liminf}]
		Without loss of generality we can assume that $\sup_n \mathcal{F}_n(z_n)\leq C$ and, arguing as in the proof of Proposition \ref{prop: cpt MM}, that the sequence $z_n$ is equibounded in $\Omega$. 
		
		First, we show that if $z_n$ converges in $L^1(\Omega;\R^N)$ to $z$, then also the affine interpolant $\tilde{z}_n$ (recall Definition \ref{affine interpolant}) converges to $z$ in $L^1(\Omega;\R^N)$ and 
		\begin{equation}\label{oooooo}
		\liminf_{n \to +\infty}  \mathcal{F}_n(z_n) \geq  \liminf_{n \to +\infty} \frac{1}{\varepsilon_n} \int_\Omega \left(|\tilde{z}_n|_A^2-1\right)^2 \, dx + \varepsilon_n \int_\Omega |\nabla \tilde{z}_n|^2 \, dx.
		\end{equation}
		Recalling that $\#R_n(\Omega) \leq C\lambda_n^{-d}$, we have 
		\begin{align*}
		\int_\Omega |\tilde{z}_n-z_n| \, dx  & \leq \sum_{i \in R_n(\Omega)} \lambda_n^d \left( \sum_{\ell=1}^d \left| z_n^i-z_n^{i+ e_\ell} \right| \right) \leq \sqrt{ C\lambda_n^{-d} \sum_{i \in R_n(\Omega)} \lambda_n^{2d} \left( \sqrt{d} \sum_{\ell=1}^d \left| z_n^i-z_n^{i+ e_\ell} \right|^2 \right)} \\
		& \leq C \lambda_n \sqrt{\sum_{i \in R_n(\Omega)} \lambda_n^d \left( \sum_{\ell=1}^d \left| \frac{z_n^i-z_n^{i+ e_\ell}}{\lambda_n} \right|^2 \right)} \leq C \frac{\lambda_n}{\sqrt{\varepsilon_n}} \leq C r_n.
		\end{align*}
		Hence, 
		$
		\Vert \tilde{z}_n-z \Vert_{L^1} \leq \Vert \tilde{z}_n -z_n \Vert_{L^1} +\Vert {z}_n-z \Vert_{L^1},
		$
		that is, $\tilde{z}_n$ converges in $L^1(\Omega;\R^N)$ to $z$ as $n \to +\infty$.
		Arguing similarly to the proof of Proposition \ref{prop: cpt MM}, for $\sigma>0$ we estimate
		\begin{align*}
			\int_\Omega (|\tilde{z}_n|_A^2-1)^2 \, dx & 
			\leq (1+3\sigma) \sum_{i \in R_n(\Omega)} \lambda_n^d (|z_n^i|_A^2-1)^2 + c_2^2C \left(1+\frac{1}{\sigma}\right) \sum_{i \in R_n(\Omega)} \lambda_n^d \left( \sum_{\ell=1}^d \left| z_n^i-z_n^{i+ e_\ell} \right|^2 \right) \\
			&  \leq (1+3\sigma) \sum_{i \in R_n(\Omega)} \lambda_n^d (|z_n^i|_A^2-1)^2  + c_2^2C \left(1+\frac{1}{\sigma}\right) \frac{\lambda_n^2}{\varepsilon_n}.
		\end{align*}
		Thus,
		\begin{align*}
				& \frac{1}{\varepsilon_n} \int_\Omega (|\tilde{z}_n|_A^2-1)^2 \, dx + \varepsilon_n \int_\Omega |\nabla \tilde{z}_n|^2 \, dx \\
				& =  \frac{1}{\varepsilon_n} \int_\Omega (|\tilde{z}_n|_A^2-1)^2 \, dx +\varepsilon_n \sum_{i \in R_n(\Omega)} \lambda_n^d \left( \sum_{\ell=1}^d \left| \frac{z_n^i-z_n^{i+ e_\ell}}{\lambda_n} \right|^2 \right)\\
				& \leq (1+3\sigma) \mathcal{F}_n(z_n)+\frac{\lambda_n^2}{\varepsilon_n^2}C\left(1+\frac{1}{\sigma}\right).
		\end{align*}
		Since $\sigma>0$ is arbitrary, we obtain the desired inequality \eqref{oooooo}. Hence, it is enough to prove that 
		$$
		\liminf_{n \to +\infty} \left( \frac{1}{\varepsilon_n} \int_\Omega \left(|\tilde{z}_n|_A^2-1\right)^2 \, dx + \varepsilon_n \int_\Omega |\nabla \tilde{z}_n|^2 \, dx \right) \geq \int_{J_z} h(z^-,z^+) \, d \mathcal{H}^{d-1}
		$$
		for $\tilde{z}_n \to z$ in $L^1(\Omega,\R^N)$.\\
		
		We now prove the statement for $d=1$ and then we proceed by a slicing argument.\\
		
		\noindent \textbf{One dimensional estimate:} We assume that the set $\Omega$ is a bounded open interval $I$ of $\R$. By Proposition \ref{prop: cpt MM}, we have that there exists $z \in BV(I,\overline{Q})$ such that u.t.s. $z_n \to z$ and $\tilde{z}_n \to z$ in $L^1(I,\R^N)$. Hence, the set $J_z$ is a discrete finite set contained in $I$. Observe that in \eqref{06}, since $d=1$ and $\frac{\lambda_n}{\varepsilon_n r_n^2} \to 0$ by assumption, we have that for $n$ large enough $\#\mathcal{I}_n=0$. Thus, for every $t \in J_z$ we have that there exists a sequence $\sigma_n \to t$ such that $\tilde{z}_n(\sigma_n) \to 0$. Fix $t_0 \in J_z$ and assume that $z=q^- \in \overline{Q}$ and $z=q^+ \in \overline{Q}$ in a left and a right neighborhood of $t_0$, respectively. There exist three sequences $\{t_n\},\{s_n\},\{\sigma_n\} \subset I$ such that $s_n<t_n$, $\sigma_n \in (s_n,t_n)$ and
		\begin{equation}\label{09}
			\tilde{z}_n(s_n) \to q^-, \qquad \tilde{z}_n(t_n) \to q^+, \qquad \tilde{z}_n(\sigma_n) \to 0.
		\end{equation}
		Fix $\eta>0$ small and let $\overline{s}_n:=\min \{ t \in [s_n,\sigma_n] \colon \ |{z}_n(t)| \leq  \eta\}$. Hence, since $\#\mathcal{I}_n=0$, there exists $j \in \{1,\dots,k_n\}$ such that for every $t \in [s_n,\overline{s}_n]$, $z_n(t) \in \mathfrak{L}^n_j$. By definition of $\tilde{z}_n$ we thus have $\tilde{z}_n(t) \in \mathfrak{L}^n_j$ for every $t \in [s_n+\lambda_n,\overline{s}_n-\lambda_n]$, that is, $\tilde{z}_n(t)=\mu_n(t)q_j^n+\delta_n(t)$ with $\mu_n(t) \in \R$ and $\delta_n(t) \in \R^N$, $|\delta_n(t)| \leq r_n$, $|q^n_j|_A=1$ and $q^n_j \to q^-$. Moreover, using again that $\#\mathcal{I}_n=0$ for $n$ large enough and the fact that $|q^-| \gg \eta$, we infer that the set $[s_n+\lambda_n,\overline{s}_n-\lambda_n]$ is non empty. Therefore, setting $\tilde{w}_n:=|\tilde{z}_n|$, and using the fact that $z_n$ is uniformly bounded on $I$ we estimate
		\begin{align*}
			\frac{1}{\varepsilon_n} & \int_{s_n+\lambda_n}^{\overline{s}_n-\lambda_n} \left(|\tilde{z}_n|_A^2-1\right)^2 \, dt + \varepsilon_n \int_{s_n+\lambda_n}^{\overline{s}_n-\lambda_n} |\nabla \tilde{z}_n|^2 \, dt \geq 2\int_{s_n+\lambda_n}^{\overline{s}_n-\lambda_n} \left(|\tilde{z}_n|_A^2-1\right)|\nabla \tilde{w}_n| \, dt \\
			& = 2\int_{s_n+\lambda_n}^{\overline{s}_n-\lambda_n} \left(|\mu_n(t)q_j^n+\delta_n(t)|_A^2-1\right)|\nabla \tilde{w}_n| \, dt \\
			& \geq 2\int_{s_n+\lambda_n}^{\overline{s}_n-\lambda_n} \left(\frac{|\tilde{z}_n-\delta_n|^2}{|q_j^n|^2}-1\right)|\nabla \tilde{w}_n| \, dt - 2C \Vert \tilde{z}_n \Vert_{L^\infty} \int_{s_n+\lambda_n}^{\overline{s}_n-\lambda_n} |\delta_n||\nabla \tilde{w}_n| \, dt \\
			& \geq 2\int_{s_n+\lambda_n}^{\overline{s}_n-\lambda_n} \left(\frac{|\tilde{w}_n|^2}{|q_j^n|^2}-1\right)|\nabla \tilde{w}_n| \, dt
			-4C \Vert \tilde{z}_n \Vert_{L^\infty} r_n \int_{s_n+\lambda_n}^{\overline{s}_n-\lambda_n} |\nabla \tilde{w}_n| \, dt \\
			& \geq 2\int_{\eta}^{|q^-|-\eta} \left( \frac{\tau^2}{|q_j^n|}-1 \right) \, d\tau-4C\Vert \tilde{z}_n \Vert_{L^\infty} r_n |q^-| \geq \frac{4}{3}|q^-|-C\eta-Cr_n-C|q_j^n-q^-|.
		\end{align*}
		Repeating the same argument for $q^+$ and then for every $t \in J_z$, we get the desired inequality
		$$
		\liminf_{n \to +\infty} \mathcal{F}_n(z_n) \geq \sum_{t \in J_z} \frac{4}{3} \left( |z^-(t)|+|z^+(t)| \right).
		$$
		\noindent \textbf{$d$-dimensional estimate:} We first fix some notations. Given $\xi \in \Sf^{N-1}$ we define $\Pi_\xi:=\{x \in \R^N \colon \ \langle z,\xi \rangle =0 \}$. For every $y \in \Pi_\xi$ we set
		$$
		\Omega_{\xi,y}:=\{ t \in \R \colon y+t\xi \in \Omega \} \qquad z_{\xi,y}(t):=z(y+t\xi).
		$$
		For every open set $U \subseteq \Omega$ we localize the functionals defining
		$$
		\mathcal{F}_n(\tilde{z},U):=\frac{1}{\varepsilon_n} \int_U (|\tilde{z}|_A^2-1)^2 \, dx + \varepsilon_n \int_U |\nabla \tilde{z}|^2 \, dx.
		$$
		Given $\xi \in \Sf^{N-1}$ and $y \in \Pi_\xi$ we also set
		$$
		\mathcal{F}_n^{\xi,y}(\tilde{z},I):=\frac{1}{\varepsilon_n} \int_I (|\tilde{z}_{\xi,y}|_A^2-1)^2 \, dx + \varepsilon_n \int_I |\tilde{z}_{\xi,y}'|^2 \, dx
		$$
		and 
		$$
		\mathcal{F}_n^\xi(\tilde{z},U):=\int_{\Pi_\xi} \mathcal{F}_n^{\xi,y}(\tilde{z},U_{\xi,y}) \, d \mathcal{H}^{d-1}(y):=\frac{1}{\varepsilon_n} \int_U (|\tilde{z}|_A^2-1)^2 \, dx + \varepsilon_n \int_U |\langle \nabla \tilde{z},\xi \rangle|^2 \, dx.
		$$
		Thus, $\mathcal{F}_n^\xi(\tilde{z},U) \leq \mathcal{F}_n(\tilde{z},U)$. Let now $z_n \in \mathcal{U}_n^d(\Omega;\R^N)$ such that $\liminf_n \mathcal{F}_n(z_n) \leq C$. Thanks to Proposition \ref{prop: cpt MM} we have that there exists $z \in BV(\Omega;\mathcal{Q})$ such that u.t.s. $z_n \to z$ and $\tilde{z}_n \to z$ in $L^1(\Omega;\R^n)$.
		Recalling the definition of $\mathcal{Q}_n$ in the proof of Proposition \ref{prop: cpt MM}, notice that for every $\xi \in \Sf^{N-1}$ we have
		$$
		\mathcal{H}^{d-1}\left(\{ y \in \Pi_\xi \colon \ \Omega_{\xi,y} \cap \mathcal{Q}_n \neq \emptyset \} \right) \leq \mathcal{H}^{d-1}(\partial \mathcal{Q}_n).
		$$ 
		Moreover, by \eqref{06} and \eqref{07} we have that
		$$
		|\mathcal{Q}_n| \leq C\frac{\lambda_n^2}{\varepsilon_n r_n^2}, \qquad \mathcal{H}^{d-1}(\partial \mathcal{Q}_n) \leq C \frac{\lambda_n}{\varepsilon_n r_n^2},
		$$
		hence by assumption $\mathcal{H}^{d-1}(\mathcal{Q}_n) \to 0$. Fix $\xi \in \Sf^{N-1}$ and observe that for $\mathcal{H}^{d-1}$-a.e. $y \in \{ y \in \Pi_\xi \colon \ \Omega_{\xi,y} \cap \mathcal{Q}_n = \emptyset \}$, since $z_n \to z$ on $\Omega_{\xi,y}$, by the one dimensional estimate we have
		$$
		\mathcal{F}^{\xi,y}({z},U_{\xi,y}):=\liminf_{n \to +\infty} \mathcal{F}_n^{\xi,y}(\tilde{z}_n,U_{\xi,y}) \geq \sum_{t \in J_z \cap U_{\xi,y}} \frac{4}{3} \left( |{z}^+|+|{z}^-| \right).
		$$
		Using the fact that $\mathcal{H}^{d-1}(\partial \mathcal{Q}_n) \to 0$ and Fatou lemma, we estimate for every $\xi \in \Sf^{N-1}$
		\begin{align*}
			\liminf_{n \to +\infty} \mathcal{F}_n(\tilde{z}_n,U) & \geq \liminf_{n \to +\infty} \mathcal{F}_n^\xi(\tilde{z}_n,U) = \liminf_{n \to +\infty} \int_{\Pi_\xi} \mathcal{F}_n^{\xi,y}(\tilde{z}_n,U_{\xi,y}) \, d \mathcal{H}^{d-1}(y) \\
			& \geq \int_{\Pi_\xi} \mathcal{F}^{\xi,y}({z},U_{\xi,y}) \, d \mathcal{H}^{d-1}(y) \geq \int_{\Pi_\xi} \sum_{t \in J_z \cap U_{\xi,y}} \frac{4}{3} \left( |{z}^+|+|{z}^-| \right) \, d \mathcal{H}^{d-1}(y) \\
			& = \int_{J_z \cap U} \frac{4}{3} \left( |z^-(x)|+|z^+(x)| \right) |\langle \nu_z(x),\xi \rangle| \, \mathcal{H}^{d-1}(x).
		\end{align*}
		Therefore, setting $\mathcal{F}^\xi({z},U):= \int_{\Pi_\xi} \mathcal{F}^{\xi,y}({z},U_{\xi,y}) \, d\mathcal{H}^{d-1}(y)$, for every $\xi \in \Sf^{N-1}$ and every open set $U \subseteq \Omega$ we have
		$$
		\liminf_{n \to +\infty} \mathcal{F}_n(\tilde{z}_n,U) \geq \mathcal{F}^\xi({z},U) \geq \int_{J_z \cap U} \frac{4}{3} \left( |z^-(x)|+|z^+(x)| \right) |\langle \nu_z(x),\xi \rangle| \, \mathcal{H}^{d-1}(x).
		$$
		Since $\mathcal{F}_n$ are local functionals, by standard arguments (see \cite[Lemma 15.2]{BraidesGamma}) we deduce that, given a dense countable subset $\{\xi_i\}$ of $\Sf^{N-1}$, for every open set $U \subseteq \Omega$ 
		$$
		\liminf_{n \to +\infty} \mathcal{F}_n(\tilde{z}_n,U) \geq \int_{J_z \cap U} \frac{4}{3} \left( |z^-(x)|+|z^+(x)| \right) \sup_i \{|\langle \nu_z(x),\xi_i \rangle|\} \, \mathcal{H}^{d-1}(x).
		$$
		Since $\sup_i \{|\langle \nu_z(x),\xi_i \rangle|\}=1$ for every $x \in J_z$, we conclude.

	\end{proof}

	The following result is contained in \cite[Theorem 2.1 and Corollary 2.4]{BraCoGarr}.
	
	\begin{thm}\label{thm: density of poly}
		Let $\mathcal{Z} \subset \R^N$ be a finite set and $\Omega \subset \R^d$ a bounded open set with Lipschitz boundary. Given $z \in SBV(\Omega;\mathcal{Z})$ there exists a sequence $z_n \in SBV(\Omega;\mathcal{Z})$ such that $J_{z_n}$ is composed by the intersection of $\Omega$ with a finite union of ($d-1$)-dimensional simplexes, $z_n \to z$ in $L^1(\Omega;\mathcal{Z})$ and $\mathcal{H}^{d-1}(J_{z_n}) \to \mathcal{H}^{d-1}(J_z)$.
	\end{thm}

	We now present an approximation result which will be useful in the proof of the limsup inequality.
	
	\begin{lem}\label{lem: density of poly}
		Let $z \in BV_{\textnormal{pc}}(\Omega;\overline{Q})$ where $\overline{Q} \subset K$ with $K \subset \R^N$ compact. Then, there exists $z_n \in BV_{\textnormal{pc}}(\Omega;\overline{Q})$ such that $J_{z_n}$ is composed by the intersection of $\Omega$ with a finite union of ($d-1$)-dimensional simplexes, $z_n \to z$ in $L^1(\Omega;\overline{Q})$ and $\mathcal{H}^{d-1}(J_{z_n}) \to \mathcal{H}^{d-1}(J_z)$.
	\end{lem}

	\begin{proof}
		If $z$ takes only a finite number of values in $\overline{Q}$ then we can use Theorem \ref{thm: density of poly} to conclude. Otherwise, let us show that for every $\varepsilon>0$ small it is always possible to find $z_\varepsilon \in BV_{\textnormal{pc}}(\Omega;\overline{Q})$ such that $z_\varepsilon$ takes only a finite number of values on $\Omega$, $\Vert z_\varepsilon-z \Vert_{L^1(\Omega)} \leq C \varepsilon$ and $\mathcal{H}^{d-1}(J_z \Delta J_{z_\varepsilon}) \leq \varepsilon$.
		
		Let $\{E_\alpha\}_\alpha$ be the Caccioppoli partition of $\Omega$ induced by $J_z$. We have $\sum_\alpha \mathcal{H}^{d-1}(\partial^* E_\alpha \setminus \partial \Omega)<+\infty$ and $\sum_\alpha |E_\alpha|=|\Omega|$. Let $\alpha_0 \geq 1$ be such that 
		$$
		\sum_{\alpha \geq \alpha_0} \mathcal{H}^{d-1}(\partial^* E_\alpha \setminus \partial \Omega) \leq \varepsilon, \qquad \sum_{\alpha \geq \alpha_0} |E_\alpha| \leq \varepsilon.
		$$
		Let $q \in \overline{Q}$, we define $z_\varepsilon=z$ on $\cup_{\alpha < \alpha_0} E_\alpha$ and $z_\varepsilon=q$ on $\cup_{\alpha \geq \alpha_0}E_\alpha$. By definition of $z_\varepsilon$ we estimate
		$$
		\mathcal{H}^{d-1}(J_z \Delta J_{z_\varepsilon}) \leq \sum_{\alpha \geq \alpha_0} \mathcal{H}^{d-1}(\partial^* E_\alpha \setminus \partial \Omega) \leq \varepsilon,\qquad \Vert z-z_\varepsilon\Vert_{L^1(\Omega)} \leq (|z|+|z_\varepsilon|) \bigcup_{\alpha \geq \alpha_0}|E_\alpha | \leq C \varepsilon.
		$$
		Thus we conclude.
	\end{proof}
We are ready to prove the limsup inequality.
\begin{proof}[Proof of Theorem \ref{limsup}]
		In view of Lemma \ref{lem: density of poly} we can assume that $J_z$ is the intersection of $\Omega$ with a finite union of ($d-1$)-dimensional simplexes. Hence, $J_z$ is relatively closed in $\Omega$ and it coincides with its Minkowski content. Fix $\tau>0$ small. For every $n \geq 1$, we define $\delta_n \in C(\Omega)$ the regularized function such that $\delta_n \in C^\infty(\Omega \setminus J_z)$, and, for every $x \in \Omega$,
		\begin{equation}\label{013}
			\frac{(\dist(x,J_z)-3\sqrt{d}\lambda_n)^+}{2} \leq \delta_n(x) \leq 2 (\dist(x,J_z)-2\sqrt{d}\lambda_n)^+, \ \ \  \Vert \nabla \delta_n \Vert_{L^\infty(\Omega)} \leq 1+\tau, \ \ \ |\nabla^2 \delta_n(x)| \leq \frac{C}{\dist(x,J_z)}.
		\end{equation}
		Let $f_\tau \in C^2([0,+\infty))$ be the almost optimal profile function such that $f_\tau(0)=0$, $f_\tau(t)=1$ for every $t \geq 1/\tau$ and
		\begin{equation}\label{014,5}
		\int_0^{+\infty} (f_\tau(t)^2-1)^2 \, dt+\int_0^{+\infty} (f_\tau'(t))^2 \, dt \leq (1+\tau)\frac{4}{3}.
		\end{equation}
		Since $z \in \overline{Q}$ in $\Omega$, there exists a sequence $\hat{z}_n \in BV(\Omega,Q_n)$ such that $J_{\hat{z}_n}=J_z$ for every $n$ and $\hat{z}_n \to z$ uniformly on $\Omega$.
		We define the sequence of functions $v_n \in C^2(\Omega)$ as 
		$$
		v_n(x):=|\hat{z}_n|f_\tau\left(\frac{\delta_n(x)}{\varepsilon_n|\hat{z}_n|}\right).
		$$
		For every $x \in \Omega \setminus J_{z}$ using \eqref{013} we estimate
		\begin{equation}\label{013.5}
			|\nabla v_n(x)| \leq \left| \frac{1}{\varepsilon_n} (1+\tau) f_\tau '\left( \frac{\delta_n(x)}{\varepsilon_n|\hat{z}_n|} \right) \right|, \ \ \ |\nabla^2 v_n(x)| \leq \frac{C}{\varepsilon_n^2}\left| f_\tau ''\left( \frac{\delta_n(x)}{\varepsilon_n|\hat{z}_n|} \right) \right| +  \frac{C}{\varepsilon_n} \frac{1}{\dist(x,J_z)}\left| f_\tau '\left( \frac{\delta_n(x)}{\varepsilon_n|\hat{z}_n|} \right) \right|.
		\end{equation}
		We claim that
		\begin{equation}\label{014}
			\limsup_{n \to +\infty} \left( \frac{1}{\varepsilon_n} \int_\Omega \left(\frac{v_n^2}{|\hat{z}_n|^2}-1\right)^2 \, dx + \varepsilon_n \int_\Omega |\nabla v_n|^2 \, dx \right) \leq (1+\tau)^2 \limsup_{n \to +\infty} \int_{J_z} \frac{4}{3} \left( |\hat{z}_n^-(x)|+|\hat{z}_n^+(x)| \right) \, d \mathcal{H}^{d-1}(x).
		\end{equation}
		We are going to prove the claim sketching the arguments as they are quite standard.
		First notice that the intersection of the interfaces in $J_z$ is a set of finite $\mathcal{H}^{d-2}$ measure. Thus, in an arbitrarily small neighborhood of this set the energy is arbitrarily small. Moreover, far from $J_z$ we have that taking $n$ large enough the energy associated to $v_n$ is arbitrarily small. Let now $T$ be an interface of $J_z$ (without loss of generality we can assume that $T$ lies on $\{x_d=0\}$). Take $\gamma>0$ such that $(T \times (-\gamma,\gamma)) \cap J_z=T$, we can always find such $\gamma$ up replacing $T$ with $T' \Subset T$ compactly contained in the topology of $\R^{d-1}$. Assume that $\hat{z}_n=q_1^n$ on $T \times (-\gamma,0)$ and $\hat{z}_n=q_2^n$ on $T \times (0,\gamma)$. We estimate for $n$ large enough
		\begin{align*}
			\frac{1}{\varepsilon_n} \int_{T \times (-\gamma,0)} & \left(\frac{v_n^2}{|\hat{z}_n|^2}-1\right)^2 \, dx + \varepsilon_n \int_{T \times (-\gamma,0)} |\nabla v_n|^2 \, dx \\ 
			& = \int_T \int_{-\gamma}^0 \frac{1}{\varepsilon_n}\left( f_\tau\left(\frac{\delta_n(x',t)}{\varepsilon_n|q_1^n|}\right)^2-1 \right)^2 + \frac{1}{\varepsilon_n} f_\tau'\left(\frac{\delta_n(x',t)}{\varepsilon_n|q_1^n|}\right)^2 |\nabla \delta_n(x',t)|^2 \, dt \, d\mathcal{H}^{d-1}(x') \\
			& \leq (1+\tau)|q_1^n| \left( \int_T \int_0^{ \frac{\delta_n(x',-\gamma)}{\varepsilon_n|q_2^n|}} \left( f_\tau(s)^2-1 \right)^2+(f'_\tau(s))^2  \, ds \, d\mathcal{H}^{d-1}(x') \right) \\
			& \leq (1+\tau)^2 |\hat{z}_n| \frac{4}{3} \mathcal{H}^{d-1}(T).
		\end{align*}
		Repeating the same argument for $T \times (0,\gamma)$ and then for every interface in $J_z$ we conclude the claim. We now estimate the difference between the energy in \eqref{014} and the discretized one. Let us define the discretized function $v_n \in \mathcal{U}_n^d(\Omega;\R^N)$ as $v_n^i:=v_n(\lambda_n i)$. For every $\sigma>0$ small, using the fact that $v_n$ is equibounded and vanishes near $J_z$, we infer
		\begin{align}
			\begin{split}\label{015}
			\frac{1}{\varepsilon_n} \sum_{i \in R_n(\Omega)} & \lambda_n^d \left(\frac{(v_n^i)^2}{|\hat{z}_n(\lambda_n i)|^2}-1\right)^2 \\
			& \leq \frac{1+3\sigma}{\varepsilon_n} \int_\Omega \left(\frac{v_n^2}{|\hat{z}_n|^2}-1\right)^2 \, dx+\frac{C}{\varepsilon_n} \left(1+\frac{1}{\sigma} \right) \sum_{i \in R_n(\Omega)} \int_{\lambda_n i+[0,\lambda_n)^d} \left| v_n-v_n(\lambda_n i) \right|^2 \, dx.
			\end{split}
		\end{align}
		Let $K_{\lambda}:=J_z+B_\lambda$ for $\lambda>0$. We now estimate the reminder term in \eqref{015}. By \eqref{013} and \eqref{013.5}
		\begin{align*}
			\begin{split}
				\frac{1}{\varepsilon_n} \sum_{i \in R_n(\Omega)} \int_{\lambda_n i+[0,\lambda_n)^d} \left| v_n(x)-v_n(\lambda_n i) \right|^2 \, dx & \leq \frac{1}{\varepsilon_n} \sum_{i \in R_n(\Omega)} \int_{(\lambda_n i+[0,\lambda_n)^d) \cap K_{\varepsilon_n/\tau}} \frac{\Vert f_\tau' \Vert_{L^\infty}^2}{\varepsilon_n^2|\hat{z}_n|^2} \Vert \nabla \delta_n \Vert_{L^\infty}^2 \lambda_n^2 \, dx \\
				& \leq C \frac{1}{\varepsilon_n} \frac{\lambda_n^2}{\varepsilon_n^2} |K_{\varepsilon_n/\tau}| \leq \frac{C \mathcal{H}^{d-1}(J_z)}{\tau} \frac{\lambda_n^2}{\varepsilon_n^2}.
			\end{split}
		\end{align*}
		Hence, since $\lambda_n/\varepsilon_n \to 0$ as $n \to +\infty$,
		\begin{equation}\label{016}
			\limsup_{n \to +\infty} \frac{1}{\varepsilon_n} \sum_{i \in R_n(\Omega)} \lambda_n^d \left(\frac{(v_n^i)^2}{|\hat{z}_n(\lambda_n i)|^2}-1\right)^2 \leq \limsup_{n \to +\infty} \frac{1+3\sigma}{\varepsilon_n} \int_\Omega \left(\frac{v_n^2}{|\hat{z}_n|^2}-1\right)^2 \, dx. 
		\end{equation}
		Similarly as before, for every $\sigma>0$ small we have
		\begin{align}\label{017}
			\begin{split}
			\varepsilon_n \sum_{i \in R_n(\Omega)} \lambda_n^d & \left( \sum_{\ell=1}^d \left| \frac{v_n^i-v^{i+ e_\ell}_n}{\lambda_n} \right|^2 \right) = \varepsilon_n \sum_{i \in R_n(\Omega)} \lambda_n^d \left( \sum_{\ell=1}^d \left| \frac{1}{\lambda_n} \int_0^{\lambda_n} \frac{\partial v_n}{\partial x_\ell}(\lambda_n i+s e_\ell) \, ds \right|^2 \right) \\
			& \leq \varepsilon_n (1+3\sigma) \int_\Omega |\nabla v_n|^2 \, dx \\
			& \ \ + \varepsilon_n \left(1+\frac{1}{\sigma} \right) \sum_{i \in R_n(\Omega)} \int_{\lambda_n i+[0,\lambda_n)^d} \left( \sum_{\ell=1}^d \left|\frac{\partial v_n}{\partial x_\ell}(x)- \frac{1}{\lambda_n} \int_0^{\lambda_n} \frac{\partial v_n}{\partial x_\ell}(\lambda_n i+s e_\ell) ds \right|^2 \right) \, dx
			\end{split}
		\end{align}
		We now estimate the remainder term in \eqref{017}. Using Jensen inequality, \eqref{013} and \eqref{013.5} we get
		\begin{align*}
			\varepsilon_n \sum_{i \in R_n(\Omega)} & \int_{\lambda_n i+[0,\lambda_n)^d} \left( \sum_{\ell=1}^d \left|\frac{\partial v_n}{\partial x_\ell}(x)- \frac{1}{\lambda_n} \int_0^{\lambda_n} \frac{\partial v_n}{\partial x_\ell}(\lambda_n i+s e_\ell) ds \right|^2 \right) \, dx \\
			& \leq \frac{\varepsilon_n}{\lambda_n} \sum_{i \in R_n(\Omega)} \int_{(\lambda_n i+[0,\lambda_n)^d) \cap K_{\varepsilon_n/\tau} \setminus K_{\sqrt{d}\lambda_n}} \left( \sum_{\ell=1}^d \int_0^{\lambda_n} \left|\frac{\partial v_n}{\partial x_\ell}(x)-  \frac{\partial v_n}{\partial x_\ell}(\lambda_n i+s e_\ell)\right|^2 ds  \right) \, dx \\
			& \leq \varepsilon_n C \sum_{i \in R_n(\Omega)} \int_{(\lambda_n i+[0,\lambda_n)^d) \cap K_{\varepsilon_n/\tau} \setminus K_{\sqrt{d}\lambda_n}} \lambda_n^2 \sup_{\lambda_n i+[0,\lambda_n)^d} |\nabla^2 v_n(x)|^2 \\
			& \leq \varepsilon_n \lambda_n^2 C \left( \frac{C}{\varepsilon_n^4} |K_{\varepsilon_n/\tau}|+\frac{1}{\varepsilon_n^2} \mathcal{H}^{d-1}(J_z)\int_{\sqrt{d}\lambda_n}^{\varepsilon_n/\tau} \frac{C}{t^2} \, dt \right) \\
			& \leq \frac{C}{\tau} \mathcal{H}^{d-1}(J_z) \left( \frac{\lambda_n^2}{\varepsilon_n^2}+ \frac{\lambda_n}{\varepsilon_n} \right).
		\end{align*}
		Therefore, 
		\begin{equation}\label{018}
			\limsup_{n \to +\infty} \ \varepsilon_n \sum_{i \in R_n(\Omega)} \lambda_n^d  \left( \sum_{\ell=1}^d \left| \frac{v_n^i-v^{i+ e_\ell}_n}{\lambda_n} \right|^2 \right) \leq \limsup_{n \to +\infty} (1+3\sigma) \varepsilon_n \int_\Omega |\nabla v_n|^2 \, dx.
		\end{equation}
		Combining \eqref{014}, \eqref{016}, \eqref{018} and the fact that $\hat{z}_n \to z$ uniformly on $\Omega$, we conclude
		\begin{equation}\label{019}
			\limsup_{n \to +\infty} \mathcal{F}_n(v_n) \leq (1+3\sigma)(1+\tau)^2 \int_{J_z} \frac{4}{3} \left( |z^-(x)|+|z^+(x)| \right) \, d \mathcal{H}^{d-1}(x).
		\end{equation}
		Notice that by definition $v_n \to |\hat{z}_n|$ in $L^1(\Omega)$. We are now ready to define the recovery sequence $z_n \in \mathcal{U}_n^d(\Omega;\R^N)$ as $z_n:=v_n \frac{\hat{z}_n}{|\hat{z}_n|}$. Notice that by construction it holds $z_n \to z$ in $L^1(\Omega;\R^N)$. Finally, since $v_n=0$ on $K_{\sqrt{d}\lambda_n}$ and $|z_n|=v_n$ by definition, we infer that for every $n \geq 1$ we have $|z_n^i-z_n^{i + e_\ell}|=|v_n^i-v_n^{i + e_\ell}|$ for every $i \in R_n(\Omega)$ and $\ell = 1,\dots,d$. Thus,
		$$
		\mathcal{F}_n(z_n)=\frac{1}{\varepsilon_n} \sum_{i \in R_n(\Omega)} \lambda_n^d \left(\frac{(v_n^i)^2}{|\hat{z}_n(\lambda_n i)|^2}-1\right)^2+ \varepsilon_n \sum_{i \in R_n(\Omega)} \lambda_n^d \left( \sum_{\ell=1}^d \left| \frac{v_n^i-v^{i+ e_\ell}_n}{\lambda_n} \right|^2 \right).
		$$
		That is,
		$$
		\limsup_{n \to +\infty} \mathcal{F}_n(z_n) \leq (1+3\sigma)(1+\tau)^2 \int_{J_z} \frac{4}{3} \left( |z^-(x)|+|z^+(x)| \right) \, d \mathcal{H}^{d-1}(x).
		$$
		By arbitrariness of $\tau>0$ and $\sigma>0$ we conclude the limsup inequality.
	\end{proof}

	\section{Chirality transition in the regime $1 \ll k_n \ll \delta_n^{-1/4}$}
	
	This section is devoted to the proof of compactness and $\Gamma$-convergence for the chirality model defined in \eqref{Hslkn}, in the case $k_n \delta_n^{1/4} \to 0$. Finally, at the end of this section we prove Theorem \ref{thm: critical case}. We recall now some important facts we are going to use. The first is a compactness result for scaled energies proven in \cite[Proposition 4.3]{CiSol}. It holds for spin variables taking values in $\Sf^1$, the adaptation for $\Sf^2$ is straightforward.
	
	\begin{prop}\label{prop: bounded energy u}
		Let $\mu_n \to 0$ and let $u_n \in C_n(I,\Sf^2)$ be a sequence such that
		\begin{equation}\label{bound u}
			\sup_n H^{sl}_n(u_n) \leq C \lambda_n \mu_n,
		\end{equation}
		for some $C>0$. Then, for all $i$ we have
		$$
		|(1-\delta_n)-(u_n^i,u_n^{i+1})| \leq C \mu_n^{1/2}.
		$$
		In particular, this implies that $(u^i_n,u^{i+1}_n) \to 1$ uniformly.
	\end{prop} 

	The next two results are proven in \cite[Proposition 3.1 and 3.2]{CiRuSol} and give useful bounds at the energy scale $\lambda_n \delta_n^{3/2}$.
	
	\begin{prop}\label{prop: bound on Hsln}
		Let $z_n $ be a sequence in $L^\infty(I,\R^3)$ such that
		\begin{equation*}
			\sup_n \frac{H^{sl}_n(z_n)}{\sqrt{2}\lambda_n \delta_n^{3/2}} \leq C < +\infty,
		\end{equation*}
		and let $u_n \in C_n(I;\Sf^2)$ be such that $z_n=T_n(u_n)$ for all $n$. Then, there exists a sequence of positive real numbers $\gamma_n \to 0$ such that for $n$ sufficiently large the following two bounds hold true:
		\begin{align}
			& \frac{H^{sl}_n(z_n)}{\sqrt{2}\lambda_n \delta_n^{3/2}} \geq \label{Hsln geq} \frac{\sqrt{2\delta_n}}{\lambda_n} \sum_{i \in R_n(I)} \lambda_n \left( \left| \frac{u^{i+1}_n-u^i_n}{\sqrt{2\delta_n}} \right|^2-1 \right)^2+\frac{\lambda_n(1-\gamma_n)}{\sqrt{2\delta_n}} \sum_{i \in R_n(I)} \lambda_n \left| \frac{z^{i+1}_n-z^i_n}{\lambda_n} \right|^2,  \\
			& \frac{H^{sl}_n(z_n)}{\sqrt{2}\lambda_n \delta_n^{3/2}} \leq \label{Hsln leq} \frac{\sqrt{2\delta_n}}{\lambda_n} \sum_{i \in R_n(I)} \lambda_n \left( \left| \frac{u^{i+1}_n-u^i_n}{\sqrt{2\delta_n}} \right|^2-1 \right)^2+\frac{\lambda_n}{\sqrt{2\delta_n}} \sum_{i \in R_n(I)} \lambda_n \left| \frac{z^{i+1}_n-z^i_n}{\lambda_n} \right|^2. 
		\end{align}
	\end{prop} 

	\begin{prop}\label{prop: z compactness}
		Assume that $\frac{\lambda_n}{\sqrt{\delta_n}} \to 0$ and let $z_n$ be a sequence in $L^\infty(I,\R^3)$ such that
		\begin{equation*}
			\sup_n \frac{H^{sl}_n(z_n)}{\sqrt{2}\lambda_n \delta_n^{3/2}} < +\infty.
		\end{equation*}
		Then, $\Vert z_n \Vert_{L^\infty}$ is equibounded and, up to subsequences, $z_n$ converges weakly* in $L^\infty(I)$ to some $z \in L^\infty(I,B_1)$. If in addition $z_n \to z$ in $L^1(I)$, then $z \in L^\infty(I,\Sf^2)$. Moreover, for every $n \in \N$ we have
		\begin{equation}\label{z bounds diff}
			\sup_i |z^{i+1}_n-z^i_n|^2 \leq 2C \sqrt{\delta_n}.
		\end{equation}
		
	\end{prop}

	\begin{rem}\label{rem: u bounds diff}
		Let $z_n$ be a sequence in $L^\infty(I,\R^3)$ verifying the bound in Proposition \ref{prop: z compactness}. By definition, for every $n \in \N$ we have that 
		$$
		\sup_i \left(1-(u^i_n,u^{i+1}_n)^2\right) = \sup_i 2\delta_n|z_n^i|^2 \leq 2\delta_n \Vert z_n \Vert_{L^\infty} \leq 2\delta_n C.
		$$
		In turn, recalling Proposition \ref{prop: bounded energy u}, we infer that
		\begin{equation}\label{u bound diff}
		\sup_i |u^{i+1}_n-u^i_n|^2 = \sup_i \left( 2-2(u_n^{i+1},u_n^i) \right) \leq 4 \delta_n C.
		\end{equation}
	\end{rem}

	The following result will be important in the asymptotic analysis of the energy. In view of this result we can treat the energy as a phase field type functional. 
	
	\begin{lem}\label{lem: dist of z}
		Assume that $\lambda_n/\sqrt{\delta_n} \to 0$ and consider a sequence $k_n$ in $\N$ such that $k_n \sqrt{\delta_n} \to 0$. Let $z_n$ be a sequence in $L^\infty(I,\R^3)$ such that
		\begin{equation}\label{0}
		H^{sl,k_n}_n(z_n) \leq C \lambda_n \delta_n^{3/2}, \qquad C>0.
		\end{equation}
		Then, 
		\begin{equation}\label{1}
			\limsup_{n \to +\infty} \, \frac{\Vert \dist(z_n,\mathfrak{L}_{k_n}) \Vert_{L^\infty}}{\sqrt{k_n} \delta_n^{1/4}} \leq 5\sqrt{C},
		\end{equation}
		where $C$ is the constant appearing in \eqref{0}.
	\end{lem}

	\begin{proof}
		By definition of $H^{sl,k_n}_n$ we have that $z_n=T_n(u_n)$ for some $u_n \in C_n(I;M_{k_n})$.
		Since the sequence $z_n$ verifies the bound in Proposition \ref{prop: z compactness} and Remark \ref{rem: u bounds diff}, we have that \eqref{z bounds diff} and \eqref{u bound diff} hold for $z_n$ and $u_n$, respectively.
		
		\noindent \textbf{Step 1:} By assumption $u_n \in C_n(I;M_{k_n})$, thus there exists  $q \in \{q_1,\dots,q_{k_n}\}$ such that $u_n^0 \in \Sf^2 \cap q^\perp$. We claim that for $n $ large enough, if $\dist(z_n^0,\textnormal{span}(q)) > 2\sqrt{2}\sqrt{C}\sqrt{k_n}\delta_n^{1/4}$, then $u^l_n \notin \Sf^2 \cap q^\perp$ for every $l=1,\dots,k_n$.
		
		By rotational invariance, it is not restrictive to assume $q=(0,0,1)$ and $u_n^0=(1,0,0)$. We first prove that it is possible to find $u_n^1 \in \Sf^2$ such that $\dist(z_n^0,\textnormal{span}(q)) > 2\sqrt{2}\sqrt{C}\sqrt{k_n}\delta_n^{1/4}$. Let $u_n^1=(u^1_{n,1};u^1_{n,2};u^1_{n,3})$. By definition of $z_n^0$ it is sufficient to have 
		\begin{equation}\label{3}
			 \frac{|u^1_{n,3}|}{\sqrt{2\delta_n}} > 2\sqrt{2}\sqrt{C}\delta_n^{1/4}\sqrt{k_n}.
		\end{equation}
		Recalling that by definition of $u_n$ and \eqref{u bound diff} it must holds
		\begin{equation}\label{3.5}
		(u^1_{n,1})^2+(u^1_{n,2})^2+(u^1_{n,3})^2=1, \qquad (u^1_{n,1}-1)^2+(u^1_{n,2})^2+(u^1_{n,3})^2 \leq 4C\delta_n.
		\end{equation}
		We observe that for instance the choice $u_{n,2}^1=0$ and $u_{n,3}^1=4\sqrt{C}\delta_n^{3/4}\sqrt{k_n}$ satisfies \eqref{3} and \eqref{3.5} when $|u_{n,3}^1| < 1$ and $|u_{n,3}^1|^2<2\delta_n C$, and these two inequalities are verified for $n$ large enough since by assumption $k_n\sqrt{\delta_n} \to 0$. Since $\dist(z_n^0,\textnormal{span}(q))> 2\sqrt{2}\sqrt{C}\sqrt{k_n}\delta_n^{1/4}$, we have
		\begin{equation}\label{3.7}
			\max_{j=1,2} |z_{n,j}^0| > 2\sqrt{C}\sqrt{k_n}\delta_n^{1/4}.
		\end{equation}
		We may assume that in \eqref{3.7} $j=2$ is the maximum and that $z^0_{n,2}=-u^1_{n,3}/\sqrt{2\delta_n} < 0$ (i.e. $u^1_{n,3}>0$). The other cases are analogous. By triangle inequality, H\"older inequality, recalling \eqref{Hsln geq} and \eqref{0}, for every $l=1,\dots,k_n$ we have 
		$$
		|z_n^l-z_n^0|^2 \leq \left( \sum_{i=0}^{k_n-1} |z_n^{i+1}-z_n^i| \right)^2 \leq k_n \sum_{i=0}^{k_n-1} |z_n^{i+1}-z_n^i|^2 \leq k_n \sum_{i \in R_n(I)} |z_n^{i+1}-z_n^i|^2 \leq 2C k_n \sqrt{\delta_n}.
		$$
		Hence, recalling \eqref{3.7},
		\begin{equation}\label{3.9}
			z^l_{n,2}< - (2-\sqrt{2})\sqrt{C}\sqrt{k_n}\delta^{1/4}_n<-\frac{\sqrt{C}\sqrt{k_n}\delta_n^{1/4}}{2} \qquad \mbox{for every $l=1,\dots,k_n$},
		\end{equation}
		which proves the claim provided $u_{n,3}^l >0$ for every $l=1,\dots,k_n$.
		We prove it by induction. For $l=1$ this is true by assumption on $u^1_{n,3}$. Take $l \in \{1,\dots,k_n-1\}$, knowing that $u^l_{n,3}>0$ we want to prove that $u^{l+1}_{n,3}>0$. By definition of $z_n^l$ we have
		\begin{equation*}
			z_n^l:=\frac{1}{\sqrt{2\delta_n}} u^l_n \times u^{l+1}_n = \frac{1}{\sqrt{2\delta_n}} (u^l_n-u^{l+1}_n) \times u^{l+1}_n = \frac{1}{\sqrt{2\delta_n}}
			\begin{pmatrix}
				u^l_{n,1}-u^{l+1}_{n,1} \\
				u^l_{n,2}-u^{l+1}_{n,2} \\
				u^l_{n,3}-u^{l+1}_{n,3}
			\end{pmatrix}
			\times
			\begin{pmatrix}
				u^{l+1}_{n,1} \\
				u^{l+1}_{n,2} \\
				u^{l+1}_{n,3}
			\end{pmatrix}.
		\end{equation*}
		Hence, using \eqref{3.9},
		\begin{equation}\label{4}
		-\frac{\sqrt{C}}{2}\sqrt{k_n}\delta_n^{1/4} > z^l_{n,2} = \frac{u^{l+1}_{n,1}(u^l_{n,3}-u^{l+1}_{n,3})}{\sqrt{2 \delta_n}}-\frac{u^{l+1}_{n,3}(u^l_{n,1}-u^{l+1}_{n,1})}{\sqrt{2 \delta_n}}.
		\end{equation}
		Recalling \eqref{u bound diff} and that $u_n^0=(1,0,0)$ together with the fact that $k_n \sqrt{\delta_n} \to 0$, using triangle inequality we infer for $n$ large enough
		\begin{equation}\label{5}
			u^{l+1}_{n,1} \geq 1-2\sqrt{C}k_n \sqrt{\delta_n} \geq \frac{1}{2}, \qquad |u^{l+1}_{n,3}| \leq 2\sqrt{C}k_n \sqrt{\delta_n}, \qquad |u^l_{n,1}-u^{l+1}_{n,1}|+ |u^l_{n,3}-u^{l+1}_{n,3}|\leq 4\sqrt{C} \sqrt{\delta_n}. 
		\end{equation}
		To prove that $u^{l+1}_{n,3}>0$ it suffice to show $u^{l+1}_{n,3} \geq u^l_{n,3}$. Assume by contradiction that $u^l_{n,3}-u^{l+1}_{n,3}>0$. Then, by \eqref{4} we have $|u^{l+1}_{n,1}(u^l_{n,3}-u^{l+1}_{n,3})| <|u^{l+1}_{n,3}(u^l_{n,1}-u^{l+1}_{n,1})|$, which by \eqref{5} implies $|z^l_{n,2}| \leq 8C\sqrt{\delta_n}k_n$. However, this is in contradiction with \eqref{4}, as $\delta_n^{1/4}\sqrt{k_n} \gg \sqrt{\delta_n}k_n$. Hence, $u^{l+1}_{n,3} \geq u^l_{n,3}>0$ and the claim is proven.
		
		\noindent \textbf{Step 2:} We are now ready to prove \eqref{1}. In particular, we will prove that for every $n$ large enough we have
		$$
		\Vert \dist(z_n,\mathfrak{L}_{k_n}) \Vert_{L^\infty} \leq 5\sqrt{C}\sqrt{k_n} \delta_n^{1/4}.
		$$
		We argue by contradiction. Assume that (possibly passing to a subsequence) there exists $i=i(n)$ such that $\dist(z_n^i,\mathfrak{L}_{k_n}) > 5\sqrt{C}\sqrt{k_n} \delta_n^{1/4}$. Let $l_i \in \{1,\dots,k_n\}$ be such that $u^i_n \in S^1_{l_i}$. By Step 1 we have that $u^{i+j}_n \notin S^1_{l_i}$ for every $j = 1,\dots,k_n$. By triangle inequality and H\"older inequality for every $j=1,\dots,k_n$ we have 
		\begin{align*}
		\dist(z_n^{i+j},\mathfrak{L}_{k_n}) & \geq \dist(z_n^i,\mathfrak{L}_{k_n})-|z_n^{i+j}-z_n^i| > 4\sqrt{C}\sqrt{k_n}\delta_n^{1/4}-\sqrt{k_n} \left(\sum_{m \in R_n(I)} |z_n^{m+1}-z_n^{m}|^2 \right)^{1/2} \\
		& \geq 5\sqrt{C}\sqrt{k_n}\delta_n^{1/4}-\sqrt{2}\sqrt{C}\sqrt{k_n}\delta_n^{1/4} \geq 2\sqrt{2}\sqrt{C}\sqrt{k_n}\delta_n^{1/4}.
		\end{align*} 
		Owing Step 1, we can now repeat the same argument with $z_n^{i+j}$ for every $j=1,\dots,k_n$. We deduce that $u^{i+j}_n \in S^1_{l_{i+j}}$ where $l_{i+j} \in \{1,\dots,k_n\}$ for every $j=0,\dots,k_n-1$ and $l_{i+j} \neq l_{i+m}$ whenever $j,m \in \{0,\dots,k_n-1\}$ and $j \neq m$. Therefore, we conclude that $u^{i+k_n}_n \notin S^1_{l_{i+j}}$ for every $j \in \{0,\dots,k_n-1\}$. That is, $u_n^{i+k_n} \notin M_{k_n}$. Since $u_n \in C_n(I;M_{k_n})$ this gives a contradiction.
		
	\end{proof}

	Before proving compactness for sequences with bounded energy, we recall an useful result.

	\begin{lem}\label{lem: finite jump set of ph field}
		Let $\{v_n\}_n$ be a sequence of functions such that $v_n \in C_n(I)$ and $v_n \geq 0$ for every $n \geq 1$. Let $\beta_n \in C_n(I)$ be a sequence converging uniformly to $1$. Assume that $\lambda_n/\sqrt{\delta_n} \to 0$ and that there exists $C>0$ with the property that for every $n \geq 1$
		\begin{equation}\label{7}
			\frac{\sqrt{2\delta_n}}{\lambda_n} \sum_{i \in R_n(I)} \lambda_n \left( \left|\beta_n  v_n^i \right|^2-1 \right)^2+\frac{\lambda_n}{2\sqrt{\delta_n}} \sum_{i \in R_n(I)} \lambda_n \left| \frac{v_n^{i+1}-v_n^i}{\lambda_n} \right|^2 \leq C.
		\end{equation}
		Then, if we set 
		\begin{equation}\label{8}
			S:=\left\{ x \in I \colon \ \mbox{there exists $x_n \to x$ such that $\ \liminf_{n \to +\infty} v_n(x_n/\lambda_n)=0$} \right\},
		\end{equation}
		we have that $\# S \leq N_C$, where $N_C>0$ is a constant depending only on $C$; and for every open set $A \Subset I \setminus S$, there exists $\eta_A >0$ such that
		$$
		\liminf_{n \to +\infty} \inf_{x \in A} v_n(x/\lambda_n) \geq \eta_A.
		$$
	\end{lem}

	\begin{proof}
		The proof follows the lines of \cite[Lemma 2.5]{CrScSo}, we only highlight the main differences. 
		
		If $\# S=1$ we conclude.
		Let $s^1<s^2 \in S$. By definition there exist sequences $s^j_n \in I$ such that $s^j_n \to s^j$, $\liminf_n v_n(s^j_n/\lambda_n)=0$ for $j=1,2$ and $s_n^1<s_n^2$ for $n$ large enough. Moreover, since by \eqref{7} we have that $v_n \to 1$ almost everywhere on $I$ as $n \to +\infty$, we infer the existence of a sequence $t_n \in I$ such that $t_n \to s^1$, $t_n \in (s_n^1,s_n^2)$ and $\liminf_n v_n(t_n/\lambda_n)=1$. Since $v_n \in C_n(I)$, let $\hat{s}_n^1$ and $\hat{t}_n$ the nodes such that $v_n(\hat{t}_n/\lambda_n)=v_n(t_n/\lambda_n)$, $v_n(\hat{s}_n^1/\lambda_n)=v_n(s_n^1/\lambda_n)$ and $\hat{s}_n^1/\lambda_n \in \Z$, $\hat{t}_n/\lambda_n \in \Z$. By definition we have $\hat{s}_n^1 < \hat{t}_n<s^2$ for every $n $ large enough. Hence, in view of the previous considerations and since $\beta_n \to 1$ uniformly, we have
		\begin{align*}
			\liminf_{n \to +\infty} & \frac{\sqrt{2\delta_n}}{\lambda_n} \sum_{i \in R_n([\hat{s}_n^1,\hat{t}_n])} \lambda_n \left( \left|\beta_n  v_n^i \right|^2-1 \right)^2+\frac{\lambda_n}{2\sqrt{\delta_n}} \sum_{i \in R_n([\hat{s}_n^1,\hat{t}_n])} \lambda_n \left| \frac{v_n^{i+1}-v_n^i}{\lambda_n} \right|^2 \\
			& \geq \liminf_{n \to +\infty} \sum_{i \in R_n([\hat{s}_n^1,\hat{t}_n])} \frac{1}{\sqrt{2}} \left( |\beta_n v_n^i|^2-1 \right)|v_n^{i+1}-v_n^i| \geq \frac{1}{4\sqrt{2}}.
		\end{align*}
		Together with \eqref{7} this gives that $\#S \leq 4\sqrt{2}C+1$. Finally, the second statement follows directly form the definition of $S$.
	\end{proof}
	
	We are in a position to prove compactness for sequences $z_n \in L^\infty(I,\R^N)$ with $\frac{1}{\lambda_n \delta_n^{3/2}} H^{sl,k_n}_n(z_n) \leq C$.

	\begin{proof}[Proof of Proposition \ref{prop: z compact < 1 quarto}]
		We divide the proof into three steps. We use the notation $\hat{z}:=z/|z|$ for $z \in \R^3 \setminus \{0\}$.
		
		\noindent \textbf{Step 1:} Assume that $|z_n^l|\geq \eta>0$ for every $l \in \{0,\dots,k_n\}$. We first prove that there exists $q \in Q_{k_n}$ and $\ell \in \{0,\dots,k_n\}$ such that 
		\begin{equation}\label{17.5}
			\left|\hat{z}_n^{\ell}-q\right| \leq \frac{C \delta_n^{1/4}}{\eta}.
		\end{equation}
		By definition of $u_n \in C_n(I;\Sf^2)$, there exist two indexes $l_1,l_2 \in \{0,\dots,k_n\}$ with $l_1<l_2$ and $q \in Q_{k_n}$ such that $u_n^{l_j} \in \Sf^2 \cap q^\perp$ for $j=1,2$. By rotational invariance it is not restrictive to assume that $q=(0,0,1)$ and $u_n^{l_1}=(1,0,0)$. We claim that there exists $\ell \in \{l_1,\dots,l_2\}$ such that $|z^{\ell}_{n,2}|\leq 4\sqrt{C}\delta_n^{1/4}$. Suppose that this is not the case. By \eqref{z bounds diff} we have that 
		\begin{equation}\label{18}
		z_{n,2}^l<-4\sqrt{C}\delta_n^{1/4} \qquad \mbox{for every $l \in \{l_1,\dots,l_2\}$}.
		\end{equation}  
		The case $z_{n,2}^l>4\sqrt{C}\delta_n^{1/4}$ is analogous. We want to show that $u^l_{n,3}>0$ for every $l \in \{l_1+1,\dots,l_2\}$. This would give a contradiction since $u_{n}^{l} \in \Sf^2 \cap q^\perp$ and thus $u_{n,3}^l=0$. Recalling the definition of $z_n$ we have
		\begin{equation}\label{19}
			-4\sqrt{C}\delta_n^{1/4} > z_{n,2}^l = \frac{u_{n,1}^l(u_{n,3}^l-u_{n,3}^{l+1})}{\sqrt{2\delta_n}}-\frac{u^{l+1}_{n,3}(u_{n,1}^l-u_{n,1}^{l+1})}{\sqrt{2\delta_n}}.
		\end{equation}
		Using \eqref{u bound diff} and that $u_n^0=(1,0,0)$, for $n$ large enough and $l \in \{1,\dots,k_n\}$ we infer that
		\begin{equation}\label{20}
			|u^{l+1}_{n,1}| \geq 1-2\sqrt{C}k_n \sqrt{\delta_n} \geq \frac{1}{2}, \qquad |u^{l+1}_{n,3}| \leq 2\sqrt{C}k_n \sqrt{\delta_n}, \qquad |u^l_{n,1}-u^{l+1}_{n,1}|+ |u^l_{n,3}-u^{l+1}_{n,3}|\leq 4\sqrt{C} \sqrt{\delta_n}. 
		\end{equation}
		If it were $u_{n,3}^l>u_{n,3}^{l+1}$ we would have $|u_{n,3}^l(u_{n,1}^l-u_{n,1}^{l+1})| \geq |u_{n,1}^l(u_{n,3}^l-u_{n,3}^{l+1})|$. That is, by \eqref{19} and \eqref{20}, $|z_{n,2}^l|\leq C k_n \sqrt{\delta_n} \ll \delta_n^{1/4}$ by assumption on $k_n$, which is a contradiction in view of \eqref{18}. Thus, $u_{n,3}^l>0$ for every $l \in \{ l_1+1,\dots,l_2 \}$. Therefore, there exists $\ell \in \{ l_1,\dots,l_2 \}$ such that 
		\begin{equation}\label{21}
			|z^{\ell}_{n,2}| \leq 4\sqrt{C}\delta_n^{1/4}.
		\end{equation} 
		For every $l \in \{l_1,\dots,l_2\}$, by \eqref{u bound diff} and the fact that $u_n^0=(1,0,0)$, it holds
		\begin{equation}\label{22}
			|z_{n,1}^l| \leq \frac{1}{\sqrt{2\delta_n}} \left( |u_{n,3}^{l+1}(u_{n,2}^l-u_{n,2}^{l+1})|+|u_{n,2}^{l+1}(u_{n,3}^l-u_{n,3}^{l+1})| \right) \leq C k_n \sqrt{\delta_n} \ll \delta_n^{1/4}.
		\end{equation} 
		Hence, for $q=\left(0,0,\sign(z_{n,3}^{\ell})\right)$, combining \eqref{21} and \eqref{22} we get
		\begin{equation}\label{23}
			\left|z_n^{\ell}-|z_{n,3}^{\ell}|q\right|^2=\left| z_{n,1}^{\ell} \right|^2+\left| z_{n,2}^{\ell} \right|^2 \leq C\sqrt{\delta_n}.
		\end{equation}
		Observe that for $n$ large enough, by assumption we have $|z_{n,3}^{\ell}| \geq \eta/2$, hence by \eqref{23}
		$$
		\left( \hat{z}_n^{\ell},q \right) = \frac{(z_n^{\ell},|z_{n,3}^{\ell}|q)}{|z_n^{\ell}||z_{n,3}^{\ell}|} \geq 1 - \frac{\left|z_n^{\ell}-|z_{n,3}^{\ell}|q\right|^2}{|z_n^{\ell}||z_{n,3}^{\ell}|} \geq 1-\frac{C\sqrt{\delta_n}}{\eta^2}.
		$$
		Finally, 
		$$
		\left|\hat{z}_n^{\ell}-q\right|^2 \leq 2-2\left(\hat{z}_n^{\ell} ,q \right) \leq \frac{C\sqrt{\delta_n}}{\eta^2},
		$$
		and \eqref{17.5} is proven. 
		
		\noindent \textbf{Step 2:} Let $l_1<l_2 \in \N$ and assume that $|z_n^l| \geq \eta$ for every $l \in \{l_1,\dots,l_2\}$. Our aim is to prove that 
		\begin{equation}\label{24}
			|\hat{z}_n^{l_1}-\hat{z}_n^l| \leq \frac{Ck_n\delta_n^{1/4}}{\eta} \qquad \mbox{for every $l \in \{l_1,\dots,l_2\}$}.
		\end{equation}
		Notice that an application of H\"older inequality would not be enough, as the number of indexes between $l_1$ and $l_2$ may be much larger than $k_n$.
		Let us split the set of indexes into sets of at most $k_n$ elements, that is $\{l_1,\dots,l_2\}=\cup_{s=1}^S N_s$ where $S \in \N$, $\#N_s=k_n$ for all $s=1,\dots,S-1$ and $\#N_S \leq k_n$. Using Step 1 we have that for every $s=1,\dots,S$ there exists $i_s \in N_s$ and $q^s \in Q_{k_n}$ such that $|\hat{z}_n^{i_s}-q^s |\leq C\delta_n^{1/4}/\eta$. We define the set $\hat{N} \subset \{l_1,\dots,l_2\}$ in the following way:
		$$
		\mbox{$l \in \hat{N}$ if there exist $s,t \in \{1,\dots,S\}$ such that $i_s<l<i_t$ and } q^s=q^t.
		$$
		Since there are at most $k_n$ different elements in $Q_{k_n}$ and in view of Step 1, we infer that 
		$$\#\left( \{l_1,\dots,l_2\} \setminus \hat{N} \right) \leq 2k_n^2+k_n \leq 3 k_n^2.$$
		Moreover, for $\hat{z}^{i_s}$ and $\hat{z}_n^{i_t}$ with $q^s=q^t$, we have $|\hat{z}^{i_s}-\hat{z}^{i_t}| \leq C\delta_n^{1/4}/\eta$.
		Therefore we estimate using H\"older inequality
		\begin{align*}
		|\hat{z}_n^{l_2}-\hat{z}_n^{l_1}|^2 & \leq \left( \sum_{i \in \{l_1,\dots,l_2\} \setminus \hat{N}} |\hat{z}_n^{i+1}-\hat{z}_n^i|+ \frac{C\delta_n^{1/4}}{\eta}k_n \right)^2 \\
		& \leq 2\left( 3k_n^2 \sum_{i \in \{l_1,\dots,l_2\} \setminus \hat{N}} |\hat{z}_n^{i+1}-\hat{z}_n^i|^2 + \frac{C\sqrt{\delta_n}}{\eta^2}k_n^2 \right).
		\end{align*}
		Using the bound on the energy, for $n$ large enough this gives
		$$
		|\hat{z}_n^{l_2}-\hat{z}_n^{l_1}|^2 \leq C k_n^2 \sum_{i=l_1}^{l_2} |\hat{z}_n^{i+1}-\hat{z}_n^i|^2 +  \frac{C\sqrt{\delta_n}}{\eta^2}k_n^2 \leq C k_n^2 \frac{C\sqrt{\delta_n}}{\eta^2}.
		$$
		Observe that the estimate does not depend on $l_2>l_1$, hence, repeating the same arguments with $l_1<l<l_2$, we conclude \eqref{24}.

		\noindent \textbf{Step 3:} For every $n \geq 1$ let us set $v_n:=|z_n|$ and $h_n:=z_n/v_n$ when $v_n \neq 0$. By definition of $z_n$ we have $v_n \in C_n(I)$, $v_n \geq 0$ and $h_n \in C_n(I;\Sf^2)$. Moreover, for every $n \geq 1$ and $i \in R_n(I)$ we have
		\begin{align*}
		|z_n^{i+1}-z_n^i|^2 & =|z_n^{i+1}|^2+|z_n^i|^2-2(z_n^{i+1},z_n^i) =(v_n^{i+1})^2+(v_n^i)^2-2v_n^{i+1} v_n^i(h_n^{i+1},h_n^i) \\ 
		& \geq (v_n^{i+1})^2+(v_n^i)^2-2v_n^{i+1} v_n^i =|v_n^{i+1}-v_n^i|^2.
		\end{align*}
		Hence, since $z_n$ satisfies the bound \eqref{0}, by Proposition \ref{prop: bound on Hsln} and Lemma \ref{lem: finite jump set of ph field} there exists $J \subset I$ and $N_C \in \N$ such that $\# J \leq N_C$ and for every open set $A \Subset I \setminus J$, there exists a constant $\eta_A>0$ such that
		\begin{equation}\label{9}
			\liminf_{n \to +\infty} \inf_{x \in A} v_n(x) \geq \eta_A.
		\end{equation}
		Let $\{I_\alpha\}_{\alpha=1}^{N_c+1}$ be the family of connected components of the set $I \setminus J$. Fix $\tilde{I} \in \{I_\alpha\}_\alpha$, we claim that $z_n \to \bar{q}$ in $L^1(\tilde{I};\R^3)$, where $\bar{q} \in \overline{Q}$. To this aim, let $A \Subset \tilde{I}$ be an open set, by \eqref{9} we can find $\eta>0$ such that $v_n(x)=|z_n|(x) \geq \eta$ for every $x \in A$. In view of Lemma \ref{lem: dist of z}, we have that for every $i \in R_n(I)$ such that $\lambda_n i \in A$ there exist $t_n^i \geq 0$ and $q^i_n \in Q_{k_n}$ such that
		\begin{equation}\label{10}
			|z_n^i-t_n^i q_n^i| \leq 5\sqrt{C} \sqrt{k_n} \delta_n^{1/4}.
		\end{equation}
		Moreover, it holds $t_n^i \geq \eta/2$. Indeed suppose by contradiction that $t_n^i<\eta/2$, then we would have
		$$
		5\sqrt{C} \sqrt{k_n} \delta_n^{1/4} \geq |z_n^i-t_n^i q_n^i| \geq |v_n^i q_n^i-t_n^iq_n^i| \geq \frac{\eta}{2},
		$$
		which is clearly impossible for $n$ large, as $k_n \delta_n^{1/4} \to 0$ by assumption. Using \eqref{10}, for every $i \in R_n(I)$ such that $\lambda_n i \in A$ we infer
		$$
		(h_n^i,q_n^i) = \frac{(v_n^i h_n^i, t_n^i q_n^i)}{v_n^i t_n^i} \geq 1-\frac{|v_n^i h_n^i- t_n^i q_n^i|^2}{v_n^i t_n^i} \geq 1-\frac{C k_n \delta_n^{1/2}}{\eta^2},
		$$
		from which we estimate
		\begin{equation}\label{11}
			|h_n^i-q_n^i|^2 \leq 2 - 2(h_n^i,q_n^i) \leq \frac{C k_n \delta_n^{1/2}}{\eta^2}.
		\end{equation}
		Fix $\tilde{i} \in R_n(I)$ (possibly depending on $n$) such that $\lambda_n \tilde{i} \in A$ and set $q_n^A:=q_n^{\tilde{i}} \in Q_{k_n} $. Using \eqref{24}, \eqref{11} and triangular inequality we estimate, for every $i \in R_n(I)$ such that $i\lambda_n \in A$,
		\begin{equation}\label{25}
			|h_n^i-q^A_n| \leq |h_n^i-h_n^{\tilde{i}}|+|h_n^{\tilde{i}}-q_n^A| \leq \frac{C\sqrt{k_n}\delta_n^{1/4}}{\eta}+\frac{Ck_n\delta_n^{1/4}}{\eta} \leq \frac{Ck_n\delta_n^{1/4}}{\eta}.
		\end{equation}
		Passing to a subsequence if necessary, we have that $q_n^A \to \bar{q} \in \overline{Q}$. Therefore, for every $i \in R_n(I)$ such that $\lambda_n i \in A$, it holds
		\begin{equation}\label{24.5}
		|h_n^i-\bar{q}|\leq |h_n^i-q_n^A|+|q_n^A-\bar{q}| \leq  \frac{C k_n \delta_n^{1/4}}{\eta}+|q_n^A-\bar{q}|.
		\end{equation}
		Hence, $h_n^i \to \bar{q}$ in $L^\infty(A;\Sf^2)$ for every $i \in R_n(I)$ such that $\lambda_n i \in A$, by assumption on $k_n$. Recalling that $v_n \to 1$ in $L^1(I)$ since $\Vert z \Vert_{L^\infty}$ is equibounded, we get that $z_n \to \bar{q}$ in $L^1(A;\R^3)$. Let now $A' \Subset \tilde{I}$ be such that $A \subset A'$. Reasoning as before we infer that for the same subsequence we have that $z_n \to \bar{q}$ in $L^1(A';\R^3)$. Hence, we conclude that $z_n \to \bar{q}$ in $L^1_{\textnormal{loc}}(\tilde{I};\R^3)$. Combining this with the fact that $\Vert z_n \Vert_{L^\infty}$ is equibounded, we have proven the claim. For every $\alpha=1,\dots,N_C+1$ we thus have shown that there exists $\bar{q}_\alpha \in \overline{Q}$ such that (passing to a further subsequence if necessary) $z_n \to \bar{q}_\alpha$ in $L^1(I_\alpha;\R^3)$. Therefore, if we define $z:=\bar{q}_\alpha$ in $I_\alpha$ for every $\alpha=1,\dots,N_C+1$ we get that $z_n \to z$ in $L^1(I;\R^3)$ and $z \in BV_{\textnormal{pc}}(I;\overline{Q})$. 
	\end{proof}

	We are now ready to prove the $\Gamma$-limit of the energies $\frac{H^{sl,k_n}_n}{\sqrt{2}\lambda_n \delta_n^{3/2}}$ in the regime $k_n \delta_n^{1/4} \to 0$.

	\begin{proof}[Proof of Theorem \ref{thm: Gamma lim < 1 quarto}]
		\textbf{Lower bound.} Without loss of generality we can consider a sequence $z_n=T_n(u_n)$ for $u_n \in C_n(I,M_{k_n})$ such that $z_n \to z$ in $L^1(I,\R^3)$ and
		$$
		\liminf_{n \to +\infty} \frac{H^{sl,k_n}_n(z_n)}{\sqrt{2}\lambda_n \delta_n^{3/2}} = \lim_{n \to +\infty} \frac{H^{sl,k_n}_n(z_n)}{\sqrt{2}\lambda_n \delta_n^{3/2}}  <+\infty.
		$$
		Thanks to Proposition \ref{prop: z compact < 1 quarto} we have that $z \in BV_{\textnormal{pc}}(I,\overline{Q})$. Hence there exist $N \in \N$ values $q_\alpha \in \Sf^2$ for $\alpha=1,\dots,N$ such that $z(I)=\{q_1,\dots,q_{N}\}$. We now show that the rescaled energy $\frac{H_n^{sl,k_n}}{\sqrt{2}\lambda_n\delta_n^{3/2}}$ along the sequence $z_n$ is asymptotically equivalent to a functional $\mathcal{F}_n$ of the type presented in Subsection \ref{susect} in the one dimensional case. For every $n$ we consider the points in $\Sf^2$ given by $\{q_1,\dots,q_{N}\}$. According to the notation in Subsection \ref{susect} (and with a slight abuse of notation) we set: 
		$$
		K=\Sf^2; \ \ \ Q_n=\{q_1,\dots,q_{N}\}; \ \ \ k_n=N; \ \ \ \lambda_n=\lambda_n, \ \ \ \varepsilon_n=\frac{\lambda_n}{\sqrt{2\delta_n}}; \ \ \ r_n=2C \sqrt{k_n}\delta_n^{1/8}.
		$$
		We claim that, setting $\mathfrak{L}^n_\alpha:=\textnormal{span}(q_\alpha)+B_{r_n}$, $\mathfrak{L}^n=\cup_{\alpha=1}^{N} \mathfrak{L}^n_\alpha$, $\beta^i_n:=\frac{2}{1+(u_n^{i+1},u_n^i)}$,
		\begin{equation*}
			g_n(i,\xi)=
			\begin{cases*}
				(|\beta_n^i \xi|^2-1)^2 & if $\xi \in \mathfrak{L}^n$, \\
				+\infty & otherwise,
			\end{cases*}
		\end{equation*}
		and
		\begin{equation*}
			\mathcal{F}_n(z_n):=\frac{1}{\varepsilon_n} \sum_{i \in R_n(I)} \lambda_n g_n(i,z_n^i)+\varepsilon_n \sum_{i \in R_n(I)} \lambda_n \left| \frac{z_n^i-z_n^{i+1}}{\lambda_n} \right|^2,
		\end{equation*}
		then,
		\begin{equation}\label{13}
			\lim_{n \to +\infty} \frac{H^{sl,k_n}_n(z_n)}{\sqrt{2}\lambda_n \delta_n^{3/2}} = \lim_{n \to +\infty} \mathcal{F}_n(z_n).
		\end{equation}
		In view of Proposition \ref{prop: bound on Hsln} and the fact that 
		$$
		\left|  \frac{u_n^{i+1}-u_n^{i}}{\sqrt{2\delta_n}}\right|^2=\frac{2}{1+(u_n^{i+1},u_n^i)}|z_n^i|^2=|\beta_n^i z_n^i|^2 ,
		$$
		we only have to prove that for $n$ large enough it holds $z_n \in \mathfrak{L}^n$ on $I$. To show that $z_n^i \in \mathfrak{L}^n$ for every $i \in R_n(I)$ when $n$ is large enough, we observe that taking $\eta=\sqrt{k_n}\delta_n^{1/8}$ in \eqref{24.5}, we infer the existence of $\alpha \in \{1,\dots,N_C\}$ (depending on $i \in R_n(I)$) such that
		$$
		\left| \frac{z_n^i}{|z_n^i|}-q_\alpha \right| \leq C\sqrt{k_n}\delta_n^{1/8} \leq r_n,
		$$
		by definition of $r_n$. Since $\beta_n \to 1$ uniformly on $I$ by Proposition \ref{prop: bounded energy u}, the energy estimate gives that $|z_n| \to 1$ almost everywhere on $I$. Hence, following Step 3 in the proof of Proposition \ref{prop: z compact < 1 quarto}, we conclude that if $|z_n^i| \geq \sqrt{k_n}\delta_n^{1/8}$ then $z_n^i \in \mathfrak{L}^n_\alpha$. Finally, if $|z_n^i| \leq \sqrt{k_n}\delta_n^{1/8} \leq r_n$ we trivially have that $z_n^i \in B_{r_n}$. Thus, $z_n^i \in \mathfrak{L}^n$ for every $i \in R_n(I)$ and the claim \eqref{13} is proven. Since by definition, for $\alpha \neq \gamma$, $\mathfrak{L}_\alpha^n \cap \mathfrak{L}_\gamma^n$ is contained a ball centered in zero shrinking as $n \to +\infty$ and since $\frac{\lambda_n}{\varepsilon_n r_n^2} \to 0$, thanks to Theorem \ref{G liminf} (with some minor modifications as $g_n$ depends also on $i$) we conclude that for every $\tau>0$ arbitrarily small
		$$
		\lim_{n \to +\infty} \frac{H^{sl,k_n}_n(z_n)}{\sqrt{2}\lambda_n \delta_n^{3/2}} = \lim_{n \to +\infty} \mathcal{F}_n(z_n) \geq \left( \liminf_{n \to +\infty} 4\int_\tau^{1-\tau} \left|1-|\beta_n t|^2\right| \, dt \right) \#S(z).
		$$
		Using the elementary inequality holding whenever $\theta \in \left[0,\frac{1}{1-\tau}\right]$ and $t \in [\tau,1-\tau]$,
		$$
		\left| 1-\theta^2 t^2 \right| = 1-\theta^2 t^2 \geq 1-\frac{1}{(1-\tau)^2}t^2,
		$$
		and since for $n$ large enough $\beta_n \leq \frac{1}{1-\tau}$, we deduce that
		$$
		\lim_{n \to +\infty} \frac{H^{sl,k_n}_n(z_n)}{\sqrt{2}\lambda_n \delta_n^{3/2}} \geq \left(  4\int_\tau^{1-\tau} 1-\frac{t^2}{(1-\tau)^2} \, dt \right) \#S(z).
		$$
		Using that $\tau>0$ is arbitrarily small, we conclude the lower bound inequality
		$$
		\lim_{n \to +\infty} \frac{H^{sl,k_n}_n(z_n)}{\sqrt{2}\lambda_n \delta_n^{3/2}} \geq \frac{8}{3} \#S(z).
		$$

		\noindent \textbf{Upper bound.} The proof follows the lines of \cite[Theorem 4.2]{CiSol}, we only sketch the main steps. Since the argument is local, we can assume that $z=q \, \chi_{[0,1/2)}+q' \chi_{(1/2,1]}$ with $q,q' \in \overline{Q}$ and $q \neq q'$. Let $\varepsilon>0$, we choose $x_\varepsilon >0$ such that $|\tanh(\pm x_\varepsilon)-\pm 1| \leq \varepsilon$ and
		$$
		\int_{|x| \leq x_\varepsilon} (|\tanh(x)|^2-1)^2+|\tanh '(x)|^2 \, dx \geq \frac{8}{3}-\varepsilon.
		$$
		Since $q,q' \in \overline{Q}$, there exist $q^1_n,q^2_n \in Q_{k_n}$ such that $q^1_n \to q$ and $q_n^2 \to q'$. We define
		\begin{equation*}
			w_n^\varepsilon(x):=
			\begin{cases*}
				|f_\varepsilon(x)|q_n^1 & if $x \leq 0$; \\
				|f_\varepsilon(x)|q_n^2 & if $x>0$.
			\end{cases*}
		\end{equation*}
		Above, $f_\varepsilon \in C^1(\R)$ is an odd function such that
		\begin{equation*}
			f_\varepsilon(x):=
			\begin{cases*}
				\tanh(x) & if $x \in [0,x_\varepsilon]$; \\
				p_\varepsilon(x) & if $x \in (x_\varepsilon,x_\varepsilon+\varepsilon) $; \\
				1 & if $x \geq x_\varepsilon+\varepsilon$;
			\end{cases*}
		\end{equation*}
		where $p_\varepsilon$ is a suitable third order interpolating polynomial with $\Vert p_\varepsilon \Vert_{L^\infty} \leq 2$. By definition there exists $C>0$ such that
		\begin{equation}\label{s1}
		\int_\R (|w_n^\varepsilon|^2-1)^2+|(w_n^\varepsilon)'|^2 \, dx \leq \frac{8}{3}+C\varepsilon.
		\end{equation}
		Let $w_{n,\varepsilon} \in C_n(I,\R^3)$ and $f_{n,\varepsilon} \in C_n(I)$ be defined as follows for every $i \in R_n(I)$
		\begin{equation}\label{s2}
			w_{n,\varepsilon}^i:=w^\varepsilon_n \left( \frac{\sqrt{2\delta_n}}{\lambda_n} \left( \lambda_n i -\frac{1}{2}\right) \right), \qquad f_{n,\varepsilon}^i:=f^\varepsilon \left( \frac{\sqrt{2\delta_n}}{\lambda_n} \left( \lambda_n i -\frac{1}{2}\right) \right).
		\end{equation}
		By construction we have that $w_{n,\varepsilon} \to z$ in $L^1(I,\R^3)$ as $n \to +\infty$ and $\varepsilon \to 0$. Setting
		$$
		i_-:=\frac{1}{\lambda_n} \left( \frac{1}{2}-\frac{\lambda_n}{\sqrt{2\delta_n}}(x_\varepsilon+\varepsilon) \right) \qquad \mbox{and} \qquad i_+:=\frac{1}{\lambda_n} \left( \frac{1}{2}+\frac{\lambda_n}{\sqrt{2\delta_n}}(x_\varepsilon+\varepsilon) \right)+1,
		$$
		then $|w_{n,\varepsilon}^i|=1$ for every $i \leq i_-$ or $i \geq i_+$. We define the angles
		$$
		\varphi_{n,\varepsilon}^i:=2 \sum_{j=0}^i \arcsin\left(\sqrt{\frac{\delta_n}{2}} |f_{n,\varepsilon}^j|\right).
		$$
		Observe that $\varphi_{n,\varepsilon}^1-\varphi_{n,\varepsilon}^0=\varphi_{n,\varepsilon}^{[1/\lambda_n]}-\varphi_{n,\varepsilon}^{[1/\lambda_n]-1}$. Let $R_n^\ell \in SO(3)$ be rotations such that $R_n^\ell (\cos(\theta),\sin(\theta),0) \in (q_n^\ell)^\perp \cap \Sf^2$ for every $\theta \in [0,2\pi]$, for $\ell=1,2$ and $R_n^1 (1,0,0) = R_n^2(1,0,0)$. We define $$
		u_{n,\varepsilon}^i:=R_n^1(\cos(\varphi^i_{n,\varepsilon}),\sin(\varphi_{n,\varepsilon}^i),0) \ \ \  \mbox{for $i \leq 0$} \qquad  \mbox{and} \qquad u_{n,\varepsilon}^i:=R_n^2(\cos(\varphi^i_{n,\varepsilon}),\sin(\varphi_{n,\varepsilon}^i),0) \ \ \ \mbox{for $i>0$}.
		$$
		Notice that by construction we have $u_{n,\varepsilon} \in \overline{\mathcal{U}}_n(I)$ and $z_{n,\varepsilon}=T_n(u_{n,\varepsilon})$. Hence, $u_{n,\varepsilon}$ is a good candidate for the recovery sequence. Thanks to \eqref{Hsln leq} we have
		\begin{equation}\label{s3}
			\frac{H^{sl,k_n}_n(z_n)}{\sqrt{2}\lambda_n \delta_n^{3/2}} \leq  \frac{\sqrt{2\delta_n}}{\lambda_n} \sum_{i \in R_n(I)} \lambda_n \left( \left| \frac{u^{i+1}_n-u^i_n}{\sqrt{2\delta_n}} \right|^2-1 \right)^2+\frac{\lambda_n}{\sqrt{2\delta_n}} \sum_{i \in R_n(I)} \lambda_n \left| \frac{z^{i+1}_n-z^i_n}{\lambda_n} \right|^2.
		\end{equation}
		We now define an auxiliary sequence of piecewise constant functions as
		\begin{equation*}
			v_{n,\varepsilon}(s):=
			\begin{cases*}
				\displaystyle \frac{u_n^{i+1}-u_n^i}{\sqrt{2\delta_n}} & if $s \in \left[ \frac{\sqrt{2\delta_n}}{\lambda_n}\left(\lambda_n i -\frac{1}{2}\right) \right., \left. \frac{\sqrt{2\delta_n}}{\lambda_n}\left(\lambda_n (i+1) -\frac{1}{2}\right) \right)$ for $i \in R_n(I)$, \\
				0 & otherwise.
			\end{cases*}
		\end{equation*}
		Notice that by construction $|v_{n,\varepsilon}(s)|=1$ whenever $|s| \geq x_\varepsilon+\varepsilon+2\lambda_n$. Moreover, since every interval $\left[ \frac{\sqrt{2\delta_n}}{\lambda_n}\left(\lambda_n i -\frac{1}{2}\right) \right., \left. \frac{\sqrt{2\delta_n}}{\lambda_n}\left(\lambda_n (i+1) -\frac{1}{2}\right) \right)$ has length $\sqrt{2\delta_n} \to 0$ and since $w_{n}^\varepsilon$ is uniformly continuous on $\R$ we infer that $|v_{n,\varepsilon}-w_{n}^\varepsilon| \to 0$ uniformly in $\R$. Thus,
		$$
		\lim_{n \to +\infty} \int_{-\infty}^{+\infty} (|v_{n,\varepsilon}(s)|^2-1)^2 \, ds = \int_{-\infty}^{+\infty} (|w_n^{\varepsilon}(s)|^2-1)^2 \, ds.
		$$
		On the other hand, via the change of variables $t-\frac{1}{2}=\frac{\lambda_n}{\sqrt{2\delta_n}}s$ we get
		\begin{align*}
			\int_{-\infty}^{+\infty} (|v_{n,\varepsilon}(s)|^2-1)^2 \, ds & \geq \int_{-\frac{\sqrt{2\delta_n}}{2\lambda_n}}^{\frac{\sqrt{2\delta_n}}{2\lambda_n}} (|v_{n,\varepsilon}(s)|^2-1)^2 \, ds = \frac{\sqrt{2\delta_n}}{\lambda_n} \int_0^1 \left( \left| v_{n,\varepsilon}\left( \frac{\sqrt{2\delta_n}}{\lambda_n}(t-\frac{1}{2}) \right)  \right|^2-1 \right)^2 \, dt \\
			& \geq \frac{\sqrt{2\delta_n}}{\lambda_n} \sum_{i \in R_n(I)} \lambda_n \left( \left| \frac{u^{i+1}_n-u^i_n}{\sqrt{2\delta_n}} \right|^2-1 \right)^2.
		\end{align*}
		Therefore,
		\begin{equation}\label{s4}
			\limsup_{n \to +\infty} \ \frac{\sqrt{2\delta_n}}{\lambda_n} \sum_{i \in R_n(I)} \lambda_n \left( \left| \frac{u^{i+1}_n-u^i_n}{\sqrt{2\delta_n}} \right|^2-1 \right)^2 \leq \int_{-\infty}^{+\infty} (|w_n^{\varepsilon}(s)|^2-1)^2 \, ds.
		\end{equation}
		We now define another sequence of piecewise constant functions as
		\begin{equation*}
			\bar{z}_{n,\varepsilon,}(s):=
			\begin{cases*}
				\displaystyle \frac{z_{n,\varepsilon}^{i+1}-z_{n,\varepsilon}^i}{\sqrt{2\delta_n}} & if $s \in \left[ \frac{\sqrt{2\delta_n}}{\lambda_n}\left(\lambda_n i -\frac{1}{2}\right) \right., \left. \frac{\sqrt{2\delta_n}}{\lambda_n}\left(\lambda_n (i+1) -\frac{1}{2}\right) \right)$ for $i \in R_n(I)$, \\
				0 & otherwise.
			\end{cases*}
		\end{equation*}
		Notice that by construction $|\bar{z}_{n,\varepsilon}(s)|=1$ whenever $|s| \geq x_\varepsilon+\varepsilon+2\lambda_n$. Moreover, since every interval $\left[ \frac{\sqrt{2\delta_n}}{\lambda_n}\left(\lambda_n i -\frac{1}{2}\right) \right., \left. \frac{\sqrt{2\delta_n}}{\lambda_n}\left(\lambda_n (i+1) -\frac{1}{2}\right) \right)$ has length $\sqrt{2\delta_n} \to 0$ and since $(w_n^\varepsilon)'$ is uniformly continuous on $\R$ we infer that $|\hat{z}_{n,\varepsilon}-(w_n^\varepsilon)'| \to 0$ uniformly in $\R$. Thus,
		$$
		\lim_{n \to +\infty} \int_{-\infty}^{+\infty} |\hat{z}_{n,\varepsilon}(s)|^2 \, ds = \int_{-\infty}^{+\infty} |(w_n^{\varepsilon})'(s)|^2 \, ds.
		$$
		On the other hand, via direct computation we infer
		\begin{align*}
			\int_{-\infty}^{+\infty} |\bar{z}_{n,\varepsilon}(s)|^2 \, ds \geq \int_{-\frac{\sqrt{2\delta_n}}{2\lambda_n}}^{\frac{\sqrt{2\delta_n}}{2\lambda_n}} |\bar{z}_{n,\varepsilon}(s)|^2 \, ds = \sum_{i \in R_n(I)} \sqrt{2\delta_n} \left| \frac{z^{i+1}_n-z^i_n}{\sqrt{2\delta_n}} \right|^2
			= \frac{\lambda_n}{\sqrt{2\delta_n}} \sum_{i \in R_n(I)} \lambda_n \left| \frac{z^{i+1}_n-z^i_n}{\lambda_n} \right|^2.
		\end{align*}
		Hence,
		\begin{equation}\label{s5}
			\limsup_{n \to +\infty} \ \frac{\lambda_n}{\sqrt{2\delta_n}} \sum_{i \in R_n(I)} \lambda_n \left| \frac{z^{i+1}_n-z^i_n}{\lambda_n} \right|^2 \leq \int_{-\infty}^{+\infty} |(w_n^{\varepsilon})'(s)|^2 \, ds.
		\end{equation}
		Combining \eqref{s1},\eqref{s3}--\eqref{s5} we conclude the upper bound
		$$
		\limsup_{n \to +\infty} \frac{H^{sl,k_n}_n(z_n)}{\sqrt{2}\lambda_n \delta_n^{3/2}}  \leq \frac{8}{3}+C\varepsilon.
		$$
		By arbitrariness of $\varepsilon>0$ we conclude the proof of the upper bound.
	\end{proof}

	\section{Chirality transition in the regime $k_n \simeq \delta_n^{-1/4}$}
	In this last Section we focus on the critical case $k_n \sim \delta^{-1/4}$ proving Theorem \ref{thm: critical case}.
	
	\begin{proof}[Proof of Theorem \ref{thm: critical case}]
		
	We start by defining suitable configurations $Q_{k_n}$ with $k_n:=2\lceil \delta_n^{-1/4} \rceil$. For every $n \in \N$ and $j,m \in \{1,\dots,k_n/2\}$ we define the "horizontal" and "vertical" unit vectors as
	\begin{equation*}
		q^h_j:=
		\frac{1}{\sqrt{1+\alpha_j^2}}
		\begin{bmatrix}
			\alpha_j \\
			1 \\
			0
		\end{bmatrix}
		\qquad 
		\mbox{and}
		\qquad
		q_m^v:=
		\frac{1}{\sqrt{1+m^2\delta_n^{1/2}}}
		\begin{bmatrix}
			0 \\
			-m\delta_n^{1/4} \\
			1
		\end{bmatrix};
	\end{equation*}
	where $\alpha_j \in [0,1]$ is chosen such that
	$$
	\Sf^2 \cap (q_j^h)^\perp = \{(\sin \theta \cos \varphi_j, \sin \theta \sin \varphi_j, \cos \theta) \colon \ \theta \in [0,\pi], \ \varphi_j \in \{ j\sqrt{2 \delta_n}, j\sqrt{2 \delta_n}+\pi\}\}.
	$$
	We set
	$$
	Q_{k_n}:= \bigcup_{j=1}^{k_n/2} q^h_j \cup \bigcup_{m=1}^{k_n/2} q^v_m,
	$$
	and notice that by definition of $Q_{k_n}$ we have 
	\begin{equation*}
		\overline{Q}=
		\left\{ 
		\frac{1}{\sqrt{1+\beta^2}} 
		\begin{bmatrix}
			0 \\
			-\beta \\
			1
		\end{bmatrix}
		\colon \ \beta \in [0,1]
		\right\} 
		\cup
		\begin{bmatrix}
			0 \\
			1 \\
			0
		\end{bmatrix}.
	\end{equation*}

	\noindent \textbf{Compactness and $\Gamma$-liminf inequality.}
		We prove the following statement:
		
		\noindent \textit{Let $z_n=T_n(u_n)$ for some $u_n \in C_n(I,M_{k_n})$ be such that
		\begin{equation}\label{c1}
			H_n^{sl,k_n}(z_n) \leq C \lambda_n \delta_n^{3/2}, \qquad C>0.
		\end{equation}
		Then, there exists $z \in BV(I,\overline{Q})$ such that u.t.s. $z_n \to z$ in $L^1(I,\R^3)$. Moreover, there exists $c>0$ such that for every $z \in BV(I,\overline{Q})$
		\begin{equation}\label{c2}
			\Gamma \mbox{-}\liminf_{n \to +\infty} \frac{H_n^{sl,k_n}}{\sqrt{2}\lambda_n\delta_n^{3/2}}(z) \geq c|Dz|(I);
		\end{equation}
		with respect to the strong $L^1$-topology.}

		We divide the proof in four steps. We use again the notation $\hat{z}:=z/|z|$ for $z \in \R^3 \setminus \{0\}$.
		
		\noindent \textbf{Step 1:} Assume that there exists $\eta>0$ such that $|z_n^l| \geq \eta>0$ for every $l \in  \{0,\dots,k_n\}$. We prove that there exists $q \in Q_{k_n}$ and an index $\ell \in \{0,\dots,k_n\}$ such that 
		\begin{equation}\label{c0}
		\left|\hat{z}_n^{\ell}-q\right| \leq \frac{25C\delta_n^{1/4}}{\eta}.
		\end{equation}
		The proof is similar to Step 1 in Proposition \ref{prop: z compact < 1 quarto}, we only highlights the main differences.
		
		Since $u_n \in C_n(I,M_{k_n})$, there exists two indexes $l_1<l_2 \in \{0,\dots,k_n\}$ and $q \in Q_{k_n}$ such that $u_n^{l_1},u_n^{l_2} \in \Sf^2 \cap q^\perp$. Without loss of generality we can assume $q=(0,0,1)$ and $u_n^{l_1}=(1,0,0)$. We show that we can find an index $\ell \in \{0,\dots,k_n\}$ with the property that $|z_{n,2}^{\ell}| \leq 8C\delta_n^{1/4}$. We argue by contradiction. Suppose that
		\begin{equation}\label{c3}
			z_{n,2}^l<-8C\delta_n^{1/4} \qquad \mbox{for every $l \in \{l_1,\dots,l_2\}$};
		\end{equation}
		the case $z_{n,2}^l > 8C\delta_n^{1/4}$ is analogous. We show that $u_{n,3}^l>0$ for every $l \in \{l_1+1,\dots,l_2\}$. This would give a contraction as $u_{n,3}^{l_2}=0$. Recalling the definition of $z_n$ we have
		\begin{equation}\label{c4}
			-8C\delta_n^{1/4} > z_{n,2}^l = \frac{u_{n,1}^{l+1}(u_{n,3}^l-u_{n,3}^{l+1})}{\sqrt{2\delta_n}}-\frac{u^{l+1}_{n,3}(u_{n,1}^l-u_{n,1}^{l+1})}{\sqrt{2\delta_n}}.
		\end{equation}
		Using \eqref{u bound diff} and that $u_n^0=(1,0,0)$, for $n$ large enough and $l \in \{1,\dots,k_n\}$ we infer that
		\begin{equation}\label{c5}
			u^{l+1}_{n,1} \geq 1-2\sqrt{C}k_n \sqrt{\delta_n} \geq \frac{1}{2}, \qquad |u^{l+1}_{n,3}| \leq 2\sqrt{C}\delta_n^{1/4}, \qquad |u^l_{n,1}-u^{l+1}_{n,1}|+ |u^l_{n,3}-u^{l+1}_{n,3}|\leq 4\sqrt{C} \sqrt{\delta_n}. 
		\end{equation}
		If it were $u_{n,3}^l>u_{n,3}^{l+1}$ we would have by \eqref{c4} and \eqref{c5}
		$$
		|z_{n,2}^l| \leq \frac{|u_{n,3}^l(u_{n,1}^l-u_{n,1}^{l+1})|}{\sqrt{2\delta_n}} \leq 4\sqrt{2}C\delta_n^{1/4},
		$$
		which is impossible in view of \eqref{c4}. Thus, $u_{n,3}^l \leq u_{n,3}^{l+1}$ and $u_n^l>0$ for every $l \in \{ l_1+1,\dots,l_2 \}$, a contradiction. Therefore, there exists $\ell \in \{ l_1,\dots,l_2 \}$ such that 
		\begin{equation}\label{c6}
			|z^{\ell}_{n,2}| \leq 8C\delta_n^{1/4}.
		\end{equation} 
		Arguing as in Step 1 in the proof of Proposition \ref{prop: z compact < 1 quarto}, we also infer $|z_{n,1}^{\ell}| \leq 8C\delta_n^{1/4}$ and, from this and \eqref{c6}, 
		$$
		\left| \hat{z}_n^{\ell}-q  \right|^2 \leq \frac{512C^2\sqrt{\delta_n}}{\eta^2},
		$$
		that is, \eqref{c0} for $q=\left(0,0,\sign(z_{n,3}^{\ell})\right)$.
		
		\noindent \textbf{Step 2:} (Improved estimate) We now improve the estimate in Step 1 showing that in \eqref{c0}, for the same $q \in Q_{k_n}$, we actually have
		\begin{equation}\label{ic0}
			\left|\hat{z}_n^{\ell}-q\right| \leq \frac{150C\sqrt{C}\delta_n^{3/8}}{\eta^2}.
		\end{equation}
		For every $l \in \{0,\dots,k_n\}$ by triangle inequality, H\"older inequality, the energy estimate \eqref{c1} and \eqref{c0} we have
		\begin{align}\label{ic1}
			\begin{split}
			|\hat{z}_n^l-q| & \leq |\hat{z}_n^l-z_n^{\ell}|+|\hat{z}_n^{\ell}-q| \leq \sqrt{k_n} \left(\sum_{l=0}^{k_n-1} |\hat{z}_n^l-\hat{z}^{l+1}_n|^2\right)^{1/2} +\frac{25C\delta_n^{1/4}}{\eta} \\
			& \leq \frac{\sqrt{2}\sqrt{C}\delta_n^{1/8}}{\eta}+\frac{25C\delta_n^{1/4}}{\eta}
			\leq \frac{2\sqrt{C}\delta_n^{1/8}}{\eta}.
			\end{split}
		\end{align}
		Since $q=\left(0,0,\sign(z_{n,3}^{\ell})\right)$ in view of \eqref{ic1} we have that $|z_{n,3}^l| \geq 1/2$ and $|z_{n,1}^l|+|z_{n,2}^l| \leq C\delta_n^{1/8}/\eta$.
		Hence, recalling that by definition it holds $\hat{z}_n^l \cdot (u_n^l-u_n^{l+1})=0$ and $|u_n^{l+1}-u_n^l| \leq \sqrt{2 \delta_n}$, \eqref{ic1} entails that 
		$$
		|u_{n,3}^{l+1}-u_{n,3}^l| \leq \frac{|u_{n,1}^{l+1}-u_{n,1}^l||\hat{z}_{n,1}|+|u_{n,2}^{l+1}-u_{n,2}^l||\hat{z}_{n,2}|}{|\hat{z}_{n,3}|} \leq \frac{16C\sqrt{\delta_n}\delta_n^{1/8}}{\eta}.
		$$
		Thus, using that $u_{n,3}^0=0$,  we have for every $l \in \{0,\dots,k_n\}$
		\begin{equation}\label{ic2}
			|u_{n,3}^{l+1}-u_{n,3}^l| \leq \frac{16C\sqrt{\delta_n}\delta_n^{1/8}}{\eta} \qquad \mbox{and} \qquad |u_{n,3}^l| \leq \frac{16C\delta_n^{1/4}\delta_n^{1/8}}{\eta}.
		\end{equation}
		Replacing the estimate in \eqref{c5} with \eqref{ic2} and reasoning as in Step 1, we get that $|z_{n,2}^{\ell}| \leq 16C\sqrt{C}\delta_n^{3/8}/\eta$. Moreover, again using \eqref{ic2} and the fact that $u_{n,2}^0=0$, we have 
		$$
		|z_{n,1}^{\ell}| \leq \frac{|u_{n,2}^{\ell+1}-u_{n,2}^{\ell}||u_{n,3}^{\ell}|+|u_{n,3}^{\ell+1}-u_{n,3}^{\ell}||u_{n,2}^{\ell}|}{\sqrt{2\delta_n}} \leq \frac{16\sqrt{2}C\sqrt{C}\delta_n^{3/8}}{\eta}+ \frac{32C\sqrt{C}\delta_n^{3/8}}{\eta} \leq \frac{64C\sqrt{C}\delta_n^{3/8}}{\eta}.
		$$
		Finally, arguing as in Step 1, we conclude \eqref{ic0}.
		
		\noindent \textbf{Step 3:} Let $F_n([x_1,x_2],\cdot)$ be the rescaled functional $\frac{H^{sl,k_n}_n}{\sqrt{2}\lambda_n\delta_n^{3/2}}$ restricted to the interval $J:=[x_1,x_2] \subseteq I$. Our aim is to prove that there exists an universal constant $c>0$ such that given $z_n=T_n(u_n)$ for some $u_n \in C_n(I,M_{k_n})$ we have
		\begin{equation}\label{c7}
			F_n(J,z_n) \geq 
			\begin{cases*}
				\displaystyle c|z_n(x_2)-z_n(x_1)| & if $|z_n(x_2)-z_n(x_1)| \geq 10\delta_n^{1/32}$, \\
				\displaystyle  \vphantom{\frac{1}{2}} 0 & otherwise.
			\end{cases*}
		\end{equation}
		We distinguish two cases. First assume that $|z_n(\overline{x})| \leq \delta_n^{1/32}$ for some $\overline{x} \in J$.
		If $|z_n(x_1)| < 5 \delta_n^{1/32}$ and $|z_n(x_2)|<5 \delta_n^{1/32}$, there is nothing to prove. Assume now that $\max \{|z_n(x_1)|,|z_n(x_2)|\} \geq 5 \delta_n^{1/32}$. Without loss of generality we can assume that $|z_n(x_2)| \geq |z_n(x_1)|$. Arguing as in Theorem \ref{thm: Gamma lim < 1 quarto}, we get
		\begin{align*}
		F_n(J,z_n) & \geq F_n([\overline{x},x_2],z_n) \geq  \int_{|z_n(\overline{x})|}^{|z_n(x_2)|/2} (1-\sigma^2) \, d \sigma \geq   \int_{|z_n(x_2)|/5}^{|z_n(x_2)|/2} \frac{3}{4} \, d \sigma \\ 
		& \geq \frac{|z_n(x_2)|}{8} \geq \frac{|z_n(x_1)|+|z_n(x_2)|}{16} \geq \frac{|z_n(x_2)-z_n(x_1)|}{16}.
		\end{align*}
		Assume now instead that $|z_n| \geq \delta_n^{1/32}$ in $J$.
		Let $i_1<i_2\dots<i_K \in \Z$ be such that $i_1\lambda_n =x_1$ and $i_K \lambda_n=x_2$. We partition the set $\{i_1,\dots,i_K\}$ as $\{i_1,\dots,i_K\}=\cup_{s=1}^{S+1} N_s$, where $\#N_s = k_n+1$ for $s=1,\dots,S$ are sets of consecutive indices and $\#N_{S+1} \leq k_n+1$. Using Step 2, for every $s=1,\dots,S$, we have that there exist $i^s \in N_s$ and an associated $q^s \in Q_{k_n}$ such that 
		\begin{equation}\label{c8}
			\left|\hat{z}_n^{i^s}-q^s\right| \leq \frac{150C\sqrt{C}\delta_n^{3/8}}{\delta_n^{1/16}} = 150C\sqrt{C}\delta_n^{5/16}.
		\end{equation}
		We also recall that for every $z_n^i$ and $z_n^j$ with $|i-j| \leq 3k_n$ using the energy bound \eqref{c1} and H\"older inequality we have
		\begin{equation}\label{c8.5}
			|z_n^i-z_n^j| \leq \sqrt{3k_n} \left( \sum_{l=i}^{i+3k_n} |z_n^{l+1}-z_n^l| \right)^{1/2} C\delta_n^{1/8}.
		\end{equation}
		By Lemma \ref{lem: dist of z} we have that for $n$ large enough 
		\begin{equation}\label{c9}
		\Vert \dist\left(z_n,\textnormal{span} (Q_{k_n})\right) \Vert_{L^\infty} \leq 8\sqrt{C} \delta_n^{1/8}.
		\end{equation}
		We now distinguish three cases: 1. $q_s \in Q_{k_h}^h$ for every $s=1,\dots,S$, 2. $q_s \in Q_{k_h}^v$ for every $s=1,\dots,S$, 3. there exist $s_1,s_2 \in \{1,\dots,S\}$ such that $q_{s_1} \in Q_{k_n}^v$ and $q_{s_2} \in Q_{k_n}^h$. Notice that in the first case, since $\diam(Q_{k_n}^h) \leq 2\delta_n^{1/4}$ and in view of \eqref{c8} and \eqref{c8.5}, there is nothing to prove. The third case cannot happen since $\dist(Q_{k_n}^h,Q_{k_n}^v) \geq \frac{1}{4}$ and we assumed that $|z_n| \geq \delta_n^{1/32}$ on $J$. Therefore, we can assume that $q^s \in Q_{k_n}^v$ for every $s=1,\dots,S$. Let $T \in \N$ be the number of distinct $q^s$ satisfying \eqref{c8} for $s=1,\dots,S$. By definition of $Q_{k_n}$ and \eqref{c9} we have $T \leq k_n/2$. There exists an ordered subset of $\{1,\dots,S\}$ given by $\{s_1 \dots s_T\}$ such that $q^{s_j} \neq q^{s_l}$ for $l \neq j$. Let $\gamma \colon [1,T] \to \R^3$ be the polygonal chain with the property that $\gamma(j)=q^{s_j}$ for $j=1,\dots,T$. Recalling the definition of $Q_{k_n}^v$ and setting $L(\gamma)$ the length of the curve $\gamma$, we estimate
		\begin{equation}\label{c10}
			L(\gamma) = \sum_{j=1}^{T-1} |q^{s_{j+1}}-q^{s_j}| \geq T \frac{\delta_n^{1/4}}{2}.
		\end{equation}
		We now introduce the set $\hat{N} \subset \ \{i_1,\dots,i_K\}$ as
		$$
		\mbox{$i \in \hat{N}$ if there exist $s,t \in \{1,\dots,S\}$ such that $i^s < i < i^t$ and $q^s=q^t$.}
		$$
		In view of the previous considerations and \eqref{c8}, there exist at most $2(k_n+1)T+\#N_{S+1} \leq 4Tk_n$ indexes in $\{i_1,\dots,i_K\} \setminus \hat{N}$.
		Combining this fact with \eqref{Hsln geq}, \eqref{c8} and \eqref{c10} we have
		\begin{align*}
			32L(\gamma) F_n(J,z_n) & \geq \frac{16L(\gamma)}{\sqrt{\delta_n}} \sum_{i=i_1}^{i_K-1} |z_n^{i+1}-z_n^i|^2 \geq 4Tk_n \sum_{i \in \{i_1,\dots,i_K\} \setminus \hat{N}} |z_n^{i+1}-z_n^i|^2 \\
			& \geq \left( \sum_{i \in \{i_1,\dots,i_K\} \setminus \hat{N}} |z_n^{i+1}-z_n^i| \right)^2 \geq \left( \sum_{i=i_1}^{i_K-1} |z_n^{i+1}-z_n^i|-300C\sqrt{C}\delta_n^{5/16}T \right)^2 \\
			& \geq \left(\sum_{j=1}^{T-1} |q^{s_{j+1}}-q^{s_j}| -450C\sqrt{C}\delta_n^{1/16} \right)^2 \geq \left( L(\gamma)-\delta_n^{1/32} \right)^{2}.
		\end{align*}
		Hence, if $L(\gamma) \geq 2\delta_n^{1/32}$, we get $32F_n(J,z_n) \geq L(\gamma)/4$. Notice that by definition of $\gamma$, \eqref{c8} and \eqref{c8.5} it holds
		$$
		L(\gamma) \geq |q^S-q^1| \geq |z_n^{i_K}-z_n^{i_1}| - 300C\sqrt{C}\delta_n^{5/16}-16\sqrt{C}\delta_n^{1/8} \geq |z_n(x_2)-z_n(x_1)| - \delta_n^{1/32}.
		$$ 
		Therefore, if $|z_n(x_2)-z_n(x_1)| \geq 3\delta_n^{1/32}$, then $L(\gamma) \geq 2\delta_n^{1/32}$ and
		\begin{equation}\label{c11}
			32F_n(J,z_n) \geq \frac{L(\gamma)}{4} \geq \frac{|z_n(x_2)-z_n(x_1)| - \delta_n^{1/32}}{4} \geq \frac{|z_n(x_2)-z_n(x_1)|}{8}.
		\end{equation}
		This concludes the proof of \eqref{c7} with $c=1/256$.
		
		\noindent \textbf{Step 4:} Recall $I=[0,1]$. Let $0=t_1,\dots,t_K,t_{K+1}= 1$ be a partition of $I$ with the property that  
		\begin{equation}\label{c12}
			t_{j+1}:=\inf \{t \in (t_j,1) \colon \ |z_n(t)-z_n(t_j)| \geq 10\delta_n^{1/32}  \} \qquad j=1,\dots,K-1.
		\end{equation}
		We define the piecewise constant function $w_n \colon I \to \R^3$ as 
		$$
		w_n(t)=z_n(t_j) \qquad t \in [t_j,t_{j+1}) \qquad j=1,\dots,K.
		$$ 
		By definition of $w_n$ and \eqref{c12} we have 
		$\Vert w_n - z_n \Vert_{L^\infty(I,\R^3)} \leq 10\delta_n^{1/32}$,
		moreover, 
		$$
		c|Dw_n|(I)= c\sum_{j=1}^{K-1} |z_n(t_{j+1})-z_n(t_j)| \leq F_n(I,z_n).
		$$
		where $c>0$ is the constant in \eqref{c7}. Since $F_n(I,z_n)$ is equibounded in $n$, the sequence $\{w_n\}_n$ is bounded in $BV(I,\R^3)$. Hence, there exists $z \in BV(I,\R^3)$ such that, up to a non relabeled subsequence, $w_n \to z$ in $L^1(I,\R^3)$. Thus, $z_n \to z$ in $L^1(I,\R^3)$, and
		$$
		c|Dz|(I) \leq \liminf_{n \to +\infty} c|Dw_n|(I) \leq \liminf_{n \to +\infty} F_n(z_n) = \liminf_{n \to +\infty} \frac{H_n^{sl,k_n}}{\sqrt{2}\lambda_n\delta_n^{3/2}}(z_n).
		$$
		Finally, owing Lemma \ref{lem: dist of z} and the fact that $|z_n| \to 1$ in $I$, we infer that $\dist(z_n,Q_{k_n}) \to 0$ in measure on $I$. Therefore, $z \in BV(I,\overline{Q})$. This concludes the proof of the compactness and the $\Gamma \mbox{-} \liminf$.\\

	\noindent \textbf{$\Gamma$-limsup inequality.} To represent points on $\Sf^2$ we use spherical coordinates where $p \in \Sf^2$ is given by
	$$
	p=(\sin \theta \cos \varphi, \sin \theta \sin \varphi, \cos \theta) \qquad \mbox{with $\varphi \in [0,2\pi)$ and $\theta \in [0,\pi]$}.
	$$
	
	The main idea is that if we want to transition from $\Sf^2 \cap (q_m^v)^\perp$ to $\Sf^2 \cap (q_{m+1}^v)^\perp$ (or equivalently, from $z_n=q_m^v$ to $z_n=q_{m+1}^v$), we can slightly modify the chirality direction using all the circles given by $\Sf^2 \cap (q^h_j)^\perp$ without slowing down in the process. This in turn will allow us to transition from two different "adjacent" chiralities paying an energy which is proportional to the distance between the two. For every $j,m=1,\dots,k_n/2$, let $p_{j,m}$ be the unique point on $\Sf^2$ given by $\Sf^2 \cap (q_j^h)^\perp \cap (q_m^v)^\perp \cap \{x_1 \geq 0\}$. We first want to find an explicit expression for the point $p_{j,m}$. Notice that at the intersection, we have
	$\cos (\theta) = m\delta_n^{1/4}\sin (\theta) \sin (j \sqrt{2\delta_n})$, that is, 
	$$
	\theta=\pi/2-\arctan(m\delta_n^{1/4}\sin(j\sqrt{2\delta_n}))=\pi/2-m\delta_n^{1/4}j\sqrt{2\delta_n}+\psi_{j,m},
	$$
	where $\psi_{j,m} \ll m\delta_n^{3/4}$ is a negligible phase.
	Hence, the intersection $p_{j,m}$ in spherical coordinates is given by
	$$
	p_{j,m}=\left( \cos (jm\sqrt{2}\delta_n^{3/4}+\psi_{j,m}) \cos(j\sqrt{2\delta_n}), \ \cos (jm\sqrt{2}\delta_n^{3/4}+\psi_{j,m}) \sin(j\sqrt{2\delta_n}), \ \sin(jm\sqrt{2}\delta_n^{3/4}+\psi_{j,m}) \right).
	$$
	Moreover, notice that by definition we have $(1,0,0) \in \Sf^2 \cap (q_m^v)^\perp$ for every $m=1,\dots,k_n/2$.
	We set $\theta_{j,m}:=m\delta_n^{1/4}j\sqrt{2\delta_n}$ and $\varphi_j:=j\sqrt{2\delta_n}$. Fix $m \in \{1,\dots,k_n/2\}$, we construct a set of points on the sphere $u_n^j$ with $j=0,\dots,k_n/2$ such that $u_n^0=(1,0,0)$, $u_n^j \in \Sf^2 \cap (q_j^h)^\perp$ for $j=1,\dots,k_n/2$, $u_{n}^{k_n/2}=p_{k_n/2,m+1}$ and the energy associated to this "path" transitioning between $\Sf^2 \cap (q_{m}^v)^\perp$ to $\Sf^2 \cap (q_{m+1}^v)^\perp$ is small. Let $f \colon [0,k_n/2] \to [0,\sqrt{2}\delta_n^{3/4}k_n/2+\psi_{k_n/2,m+1}]$ be a smooth monotone increasing function with $\Vert f'' \Vert_{L^\infty} \leq 8\delta_n$, such that $f(0)=f'(0)=0$, $f(1)=\psi_{1,m}$,
	\begin{equation*} f\left(\frac{k_n}{2}-1\right)=\sqrt{2}\delta_n^{3/4}\left(\frac{k_n}{2}-1\right)+\psi_{\frac{k_n}{2}-1,m+1} \qquad \mbox{and} \qquad f\left(\frac{k_n}{2}\right)=\sqrt{2}\delta_n^{3/4}\frac{k_n}{2}+\psi_{\frac{k_n}{2},m+1}.
	\end{equation*}
	We define $u_n^j$ as
	\begin{equation}\label{e1}
		u_n^j:=\left( \cos (\theta_{j,m}+f(j)) \cos(\varphi_j), \ \cos (\theta_{j,m}+f(j)) \sin(\varphi_j), \ \sin(\theta_{j,m}+f(j)) \right).
	\end{equation}

	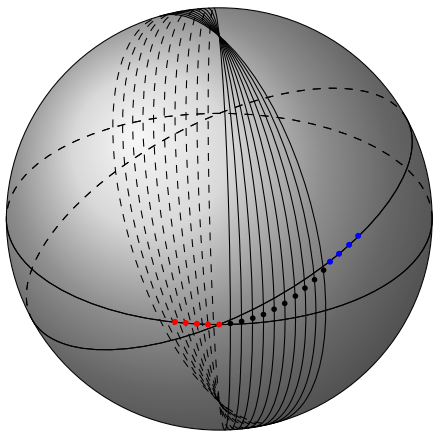
\begin{figure}[H]
		\tdplotsetmaincoords{60}{125}  
		\begin{tikzpicture}[scale=0.7]
			\draw[tdplot_screen_coords, ball color = black!20, opacity=0.4] (0,0,0) circle (4); 
			
			\foreach \angle in {128,131,...,155} {
				\scircle{4}{\angle}{90}{0};
			}
			
			\scircle{4}{0}{0}{0}
			\scircle{4}{125}{-25}{0}
			\scircle{4}{0}{0}{0}
			\scircle{4}{125}{-25}{0}
			
			\fill[red] (0,-2) circle (1.6pt); 
			\fill[red] (-0.21,-2) circle (1.6pt);
			\fill[red] (-0.42,-1.985) circle (1.6pt);
			\fill[red] (-0.63,-1.97) circle (1.6pt);
			\fill[red] (-0.83,-1.95) circle (1.61pt);
			\fill[black] (0.21,-1.98) circle (1.5pt);
			\fill[black] (0.42,-1.94) circle (1.5pt);
			\fill[black] (0.63,-1.88) circle (1.5pt);
			\fill[black] (0.83,-1.81) circle (1.5pt);
			\fill[black] (1.03,-1.71) circle (1.5pt);
			\fill[black] (1.23,-1.6) circle (1.5pt);
			\fill[black] (1.43,-1.46) circle (1.5pt);
			\fill[black] (1.61,-1.31) circle (1.5pt);
			\fill[black] (1.79,-1.15) circle (1.5pt);
			\fill[black] (1.96,-0.97) circle (1.5pt);
			\fill[blue] (2.08,-0.81) circle (1.6pt);
			\fill[blue] (2.25,-0.66) circle (1.6pt);
			\fill[blue] (2.44,-0.49) circle (1.6pt);
			\fill[blue] (2.61,-0.32) circle (1.6pt);
		\end{tikzpicture}
		\caption{Spin transition between two different chiralities (red and blue) using a family of meridians}
		\label{fig: recovery}
	\end{figure}

	By construction $u^1_n \in \Sf^2 \cap (q_m^v)^\perp$ and $u_n^{k_n/2-1} \in \Sf^2 \cap (q_{m+1}^v)^\perp$. Thus, $z_n^0=q_{m}^v$ and $z_n^{k_n/2-1}=q_{m+1}^v$.
	We claim that there exists $C>0$ such that
	$$
	\frac{\sqrt{2\delta_n}}{\lambda_n} \sum_{j=1}^{k_n/2-1} \lambda_n \left( \left| \frac{u_n^{j+1}-u_n^j}{\sqrt{2\delta_n}} \right|^2-1 \right)^2 + \frac{\lambda_n}{\sqrt{2\delta_n}} \sum_{j=1}^{k_n/2-1} \lambda_n \left| \frac{z_n^{j+1}-z_n^{j}}{\lambda_n} \right|^2 \leq C\delta_n^{1/4}. 
	$$
	We start by estimating the quantity $|u_n^{j+1}-u_n^j|$ for $j=1,\dots,k_n/2-1$. Recalling that $\varphi_j=j\sqrt{2\delta_n}$ and $\theta_{j,m}=m\delta_n^{1/4}j\sqrt{2\delta_n}$, for $\delta_n$ small enough by prosthaphaeresis formulas we have
	\begin{align}\label{e2}
		& |\cos(\varphi_{j+1})-\cos(\varphi_j)| \leq 2 \sin((2j+1)\sqrt{2\delta_n})\sin(\sqrt{2\delta_n}/2) \leq 4j\delta_n; \\
		& |\sin(\varphi_{j+1})-\sin(\varphi_j)| = 2\cos\left(  \label{e3} \frac{2j+1}{2}\sqrt{2\delta_n}\right)\sin\left(\frac{\sqrt{2\delta_n}}{2}\right) \leq \sqrt{2\delta_n}.
	\end{align}
	Moreover, since $f'(0)=0$ and $\Vert f'' \Vert_{L^\infty} \leq 8\delta_n$, we infer that $f(j) \leq 4j\delta_n^{3/4}$ and $f(j+1)-f(j) \leq 8\delta_n^{3/4}$. Thus,
	\begin{align}\label{e4}
		& \left|\cos(\theta_{j+1,m}+f(j+1))-\cos(\theta_{j,m}+f(j)) \right| \leq 16(j+1)m^2\delta_n^{3/2}; \\
		& \left|\sin(\theta_{j+1,m}+f(j+1))-\sin(\theta_{j,m}+f(j)) \right| \leq 16m\delta_n^{3/4}. \label{e5}
	\end{align}
	Combining \eqref{e2}, \eqref{e4} and using that $m,j \leq \delta_n^{-1/4}$ we get
	\begin{align}\label{e6}
		\begin{split}
			& |u^{j+1}_{n,1}-u^j_{n,1}| \\
			& \ \ \leq |\cos(\theta_{j,m}+f(j))-\cos(\theta_{j+1,m}+f(j+1))|\cos(\varphi_{j+1})|+|\cos(\varphi_{j+1})-\cos(\varphi_{j})||\cos(\theta_{j,m}+f(j))| \\
			& \ \ \leq 4j\delta_n+16(j+1)m^2\delta_n^{3/2} \\
			& \ \ \leq 20 \delta_n^{3/4}.
		\end{split}
	\end{align}
	For the second component reasoning as above, combining \eqref{e2} and \eqref{e4}, we first estimate
	\begin{align}\label{e7}
		\begin{split}
			|u^{j+1}_{n,2}-u^j_{n,2}| \leq \sqrt{2\delta_n}+16 \sqrt{2}(j+1)^2m^2 \delta_n^2 + \delta_n \leq \sqrt{2 \delta_n}+ 32 \delta_n.
		\end{split}
	\end{align}
	Observing that by \eqref{e2}
	$$
	\sin(\varphi_{j+1})-\sin(\varphi_{j}) \geq \left( 1-2(j+1)^2\delta_n \right)\left( \sqrt{2\delta_n}-2\delta_n^{3/2} \right) \geq \sqrt{2\delta_n}-2\delta_n,
	$$
	we also obtain the lower bound
	\begin{align}\label{e8}
		\begin{split}
			|u^{j+1}_{n,2} & -u^j_{n,2}| \\
			& \geq \cos(\theta_{j,m}+f(j))\left( \sin(\varphi_{j+1})-\sin(\varphi_{j}) \right)-|\sin(\varphi_{j+1})|\left|\cos(\theta_{j+1,m}+f(j+1))-\cos(\theta_{j,m}+f(j))\right| \\
			& \geq (1-2j^2 m^2 \delta_n^{3/2}-2f(j)^2)(\sqrt{2\delta_n}- 2\delta_n)-(j+1)^2 \sqrt{2 \delta_n} 16 m^2 \delta_n^{3/2} \\
			& \geq (\sqrt{2\delta_n} -2\delta_n)(1-4\sqrt{\delta_n})-16\delta_n \\
			& \geq \sqrt{2\delta_n} - 32 \delta_n.
		\end{split}
	\end{align}
	Finally, using \eqref{e5}, for the third component we have
	\begin{align}\label{e9}
		\begin{split}
			|u^{j+1}_{n,3}-u^j_{n,3}| \leq 16m\delta_n^{3/4} \qquad \mbox{for every $m=1,\dots,k_n/2$}.
		\end{split}
	\end{align}
	Combining \eqref{e6}--\eqref{e9}, given $m \in \{1,\dots,k_n/2\}$, for every $j=1,\dots,k_n/2$ we estimate
	$$
	\left( 1-32\sqrt{\delta_n} \right)^2-1 \leq \left| \frac{u_n^{j+1}-u_n^j}{\sqrt{2\delta_n}} \right|^2-1 \leq \left( 20 \delta_n^{1/4} \right)^2 + \left( 1+32\sqrt{\delta_n} \right)^2+\left( 16m\delta_n^{1/4} \right)^2-1.
	$$
	Therefore,
	$$
	\left( \left| \frac{u_n^{j+1}-u_n^j}{\sqrt{2\delta_n}} \right|^2-1 \right)^2 \leq \left( C m^2 \sqrt{\delta_n} \right)^2 \leq C m^4 \delta_n,
	$$
	where $C>0$ is a constant independent of $n$ and $m$. Hence, for every $m \in \{1,\dots,k_n/2\}$, the first part of the energy to transition from $q_m^v$ to $q_{m+1}^v$ can be estimated as
	\begin{equation}\label{e10}
		\frac{\sqrt{2\delta_n}}{\lambda_n} \sum_{j=1}^{k_n/2-1} \lambda_n \left( \left| \frac{u_n^{j+1}-u_n^j}{\sqrt{2\delta_n}} \right|^2-1 \right)^2 \leq \sqrt{2\delta_n} k_n C m^4 \delta_n \leq C\delta_n^{1/4}.
	\end{equation}
For the second term of the energy we start noticing that setting $z_n^j:=(u_n^{j} \times u_n^{j+1})/\sqrt{2\delta_n}$, we have
	\begin{equation}\label{e10.5}
	z_n^{j}-z_n^{j-1}= \frac{u^j_n \times (u_n^{j+1}+u_n^{j-1}-2u_n^j)}{\sqrt{2\delta_n}}.
	\end{equation}
	Moreover, the following estimates hold for every $m,j =1,\dots,k_n/2$ by definition of $u_n^j$
	\begin{equation}\label{e11}
		|u_{n,1}^j| \leq 1, \qquad |u_{n,2}^j| \leq j\sqrt{2\delta_n} \leq 2\delta_n^{1/4}, \qquad |u_{n,3}^j| \leq 2jm\delta_n^{3/4} \leq 2\delta_n^{1/4}.
	\end{equation}
	We are going to use the following fact: let $\phi \in C^2(\R)$ and let $x,y,z \in \R$ with $x<y<z$, then
	\begin{equation}\label{e12}
		|\phi(z)+\phi(x)-2\phi(y)| \leq \Vert \phi' \Vert_{L^\infty([x,z])}|z+x-2y|+\Vert \phi'' \Vert_{L^\infty([x,z])}|z-x|^2.
	\end{equation}
	The first component of $u_n^{j+1}+u_n^{j-1}-2u_n^j$ can be estimated using $\phi(t)=\cos(t)$ in \eqref{e12} and the fact that $\varphi_{j+1}+\varphi_{j-1}-2\varphi_j=0$ as
	\begin{align*}
		& \left| \cos (\theta_{j+1,m}+f(j+1)) \cos(\varphi_{j+1})+\cos (\theta_{j-1,m}+f(j-1)) \cos(\varphi_{j-1})-2\cos (\theta_{j,m}+f(j)) \cos(\varphi_j) \right| \\
		& \leq \left| \cos(\varphi_{j+1})+\cos(\varphi_{j-1})-2\cos(\varphi_{j}) \right|+\left| \cos (\theta_{j+1,m}+f(j+1))-\cos (\theta_{j,m}+f(j)) \right| \\
		& \qquad \qquad \qquad \qquad \qquad \qquad  \qquad \qquad \ \ + \left| \cos (\theta_{j,m}+f(j))-\cos (\theta_{j-1,m}+f(j-1)) \right| \\
		& \leq \left| \varphi_{j+1}-\varphi_{j-1} \right|^2+\sin(\theta_{j+1,m}+f(j+1)) \left| \theta_{j+1,m}+f(j+1)-\theta_{j,m}-f(j) \right| \\
		& \qquad \qquad \qquad \qquad \qquad \qquad  \qquad \qquad \ \ + \sin(\theta_{j,m}+f(j)) \left| \theta_{j,m}+f(j)-\theta_{j-1,m}-f(j-1) \right| \\
		& \leq 8\delta_n+8m(j+1)\delta_n^{3/4}m\delta_n^{3/4} \leq 8\delta_n^{3/4}.
	\end{align*}
	Hence,
	\begin{equation}\label{e13}
		\left|u_{n,1}^{j+1}+u_{n,1}^{j-1}-2u_{n,1}^j\right| \leq 16\delta_n^{3/4}.
	\end{equation}
	For the second component, using $\phi(t)=\sin(t)$ in \eqref{e12} and arguing as above we get
	\begin{align*}
		& \left| \cos (\theta_{j+1,m}+f(j+1)) \sin(\varphi_{j+1})+\cos (\theta_{j-1,m}+f(j-1)) \sin(\varphi_{j-1})-2\cos (\theta_{j,m}+f(j)) \sin(\varphi_j) \right| \\
		& \leq \left| \sin(\varphi_{j+1})+\sin(\varphi_{j-1})-2\sin(\varphi_{j}) \right|+\left| \cos (\theta_{j+1,m}+f(j+1))-\cos (\theta_{j,m}+f(j)) \right|\sin(\varphi_{j+1}) \\
		& \qquad \qquad \qquad \qquad \qquad \qquad  \qquad \qquad \ + \left| \cos (\theta_{j,m}+f(j))-\cos (\theta_{j-1,m}+f(j-1)) \right| \sin(\varphi_{j-1}) \\
		& \leq \sin(\varphi_{j+1})\left| \varphi_{j+1}-\varphi_{j-1} \right|^2+\sin(\theta_{j+1,m}+f(j+1)) \left| \theta_{j+1,m}+f(j+1)-\theta_{j,m}-f(j) \right|\sin(\varphi_{j+1}) \\
		& \qquad \qquad \qquad \qquad \qquad \qquad  \qquad \qquad \ + \sin(\theta_{j,m}+f(j)) \left| \theta_{j,m}+f(j)-\theta_{j-1,m}-f(j-1) \right|\sin(\varphi_{j-1}) \\
		& \leq 8(j+1)\delta_n^{3/2}+8m(j+1)\delta_n^{3/4}m\delta_n^{3/4}(j+1)\sqrt{2\delta_n} \leq 16\delta_n;
	\end{align*}
	that is,
	\begin{equation}\label{e14}
		\left|u_{n,2}^{j+1}+u_{n,2}^{j-1}-2u_{n,2}^j\right| \leq 16\delta_n.
	\end{equation}
	Finally, for the third component, using again $\phi(t)=\sin(t)$ in \eqref{e12} and $\Vert f'' \Vert_{L^\infty} \leq 8\delta_n$, we estimate
	\begin{align*}
		& \left| \sin(\theta_{j+1,m}+f(j+1))+\sin(\theta_{j-1,m}+f(j-1))-2\sin(\theta_{j,m}+f(j)) \right| \\
		& \leq \cos(\theta_{j-1,m}+f(j-1))\left| f(j+1)+f(j-1)-2f(j) \right| \\
		& \qquad \qquad \qquad +\sin(\theta_{j+1,m}+f(j+1))\left| \theta_{j+1,m}-\theta_{j-1,m}+f(j+1)-f(j) \right|^2 \\
		& \leq 32\delta_n+64m(j+1)\delta_n^{3/4}m^2\delta_n^{3/2} \leq 33\delta_n.
	\end{align*}
	Thus,
	\begin{equation}\label{e15}
		\left|u_{n,3}^{j+1}+u_{n,3}^{j-1}-2u_{n,3}^j\right| \leq 33\delta_n.
	\end{equation}
	Combining \eqref{e10.5}--\eqref{e15} we get that there exists $C>0$ not depending on $j$,$m$ and $n$ such that
	$$
	|z_n^j-z_n^{j-1}| \leq C\sqrt{\delta_n}.
	$$
	Therefore, for the second term of the energy, for every $m=1,\dots,k_n/2$ we have
	\begin{equation}\label{e16}
		\frac{\lambda_n}{\sqrt{2\delta_n}} \sum_{j=1}^{k_n/2-1} \lambda_n \left| \frac{z_n^{j+1}-z_n^{j}}{\lambda_n} \right|^2 \leq Ck_n \sqrt{\delta_n} \leq C \delta_n^{1/4},
	\end{equation}
	and the claim is proven.

	Let now $m_1<m_2 \in \{1,\dots,k_n/2\}$. By definition of $q_m^v$ there exist $\alpha_1,\alpha_2>0$ such that
	$$
	\alpha_1(m_2-m_1) \delta_n^{1/4} \leq |q_{m_2}^v-q_{m_1}^v| \leq \alpha_2 (m_2-m_1) \delta_n^{1/4}.
	$$
	We construct a path in order to transition from $q_{m_1}^v$ and $q_{m_2}^v$. Let $u_n^0=(1,0,0)$ and $z_n^0=q_{m_1}^v$. We define $u_n^j$ for every $j=1,\dots,k_n/2$ as the path to transition from $q_{m_1}^v$ to $q_{m_1+1}^v$ constructed above. Since $u_n^{k_n/2} \in \Sf^2 \cap (q_{m_1+1})^\perp$ and $z_n^{k_n/2-1} = q_{m_1+1}^v$, we pick $u_n^{k_n/2+1} \in \Sf^2 \cap (q_{m_1+1})^\perp$ such that
	$$
	|u_n^{k_n/2+1}-u_n^{k_n/2}| = \sqrt{2\delta_n}.
	$$
	Next, we complete the revolution around the sphere remaining constrained on $\Sf^2 \cap (q_{m_1+1})^\perp$ in a number of steps which is proportional to $\delta_n^{-1/2}$ and paying (almost) zero energy in the process. That is, we define $u^j_n \in \Sf^2 \cap (q_{m_1+1})^\perp$ for every $j \in \{k_n/2+2,\dots,k_n/2+N_n\}$, where $N_n \simeq \delta_n^{-1/2}$, such that $|u_n^j-u_n^{j-1}| = \sqrt{2\delta_n}$ for every $j \in \{k_n/2+2,\dots,k_n/2+N_n\}$ and $u_n^{k_n/2+N_n+1}=(1,0,0)$ such that 
	$$
	|u_n^{N_n+k_n/2+1}-u_n^{N_n+k_n/2}| \leq \sqrt{2\delta_n}.
	$$
	Thus, for the revolution around the sphere we paid at most $\sqrt{2\delta_n}$ in the last step, which is negligible with respect to $\delta_n^{1/4}$.
	We now have $z_n^{N_n+k_n/2}=q_{m_1+1}$ and $u_n^{k_n/2+N_n+1}=(1,0,0)$, therefore, we perform the same construction as before to transition from $q_{m_1+1}^v$ to $q_{m_1+2}^v$. Notice that the energy associated to this path is relevant only in the $k_n$ steps we are using to transition between the two chiralities and not when we are revolving around the sphere.
	Hence, using , \eqref{e10} and \eqref{e16} we can estimate the energy associated to the path we constructed in order to transition from $q_{m_1}^v$ to $q_{m_2}^v$ and get
	\begin{align}
		\begin{split}\label{fine}
	\sum_{m=m_1}^{m_2} \left( \frac{\sqrt{2\delta_n}}{\lambda_n} \sum_{j=1}^{k_n/2-1} \lambda_n \left( \left| \frac{u_n^{j+1}-u_n^j}{\sqrt{2\delta_n}} \right|^2-1 \right)^2+ \frac{\lambda_n}{\sqrt{2\delta_n}} \sum_{j=1}^{k_n/2-1} \lambda_n \left| \frac{z_n^{j+1}-z_n^{j}}{\lambda_n} \right|^2 \right) & \leq C (m_2-m_1) \delta_n^{1/4} \\ & \leq C|q_{m_2}^v-q_{m_1}^v|.
		\end{split}
	\end{align}
	Moreover, observe that in the construction we have used $\simeq \delta_n^{-1/2} \delta_n^{-1/4}$ points, hence, since by assumption $\frac{\lambda_n}{\delta_n^{3/4}} \to 0$, this construction can always be employed locally near the jump points of $z$.
	Estimate \eqref{fine} together with Proposition \ref{prop: bound on Hsln} and the fact that the construction of the recovery sequence is local, implies that 
	given $z \in BV_{\textnormal{pc}}(I,\overline{Q} \setminus \{(0,1,0)\})$ we have
	$$
	\Gamma \mbox{-} \limsup_{n \to +\infty} \frac{H^{sl,k_n}_n(z)}{\sqrt{2}\lambda_n \delta_n^{3/2}}\leq C |Dz|(I).
	$$

	Let now $z \in BV_{\textnormal{pc}}(I,\overline{Q})$. By arguing locally we can always assume that $z$ has only one jump point in $I$ and that $z$ takes two values $q \in \overline{Q} \setminus \{(0,1,0)\}$ and $(0,1,0)$. Using the construction for the upper bound proof in Theorem \ref{thm: Gamma lim < 1 quarto}, we get that 
	$$
	\Gamma \mbox{-} \limsup_{n \to +\infty} \frac{H^{sl,k_n}_n(z)}{\sqrt{2}\lambda_n \delta_n^{3/2}} \leq \frac{8}{3} \leq 8 \, \dist(\overline{Q} \setminus \{(0,1,0)\},\{(0,1,0)\}) \leq 8|Dz|(I).
	$$
	Therefore, for every $z \in BV_{\textnormal{pc}}(I,\overline{Q})$ it holds
	$$
	\Gamma \mbox{-} \limsup_{n \to +\infty} \frac{H^{sl,k_n}_n(z)}{\sqrt{2}\lambda_n \delta_n^{3/2}}\leq C |Dz|(I).
	$$
	
	Finally, for a general $z \in BV(I,\overline{Q})$ we conclude by density of $BV_{\textnormal{pc}}(I,\overline{Q})$ in $BV(I,\overline{Q})$ with respect to the total variation norm.
	\end{proof}

	\section*{Acknowledgments}
	
	This work was supported by the Italian Ministry of Education and Research through the PRIN 2022 project ``Variational Analysis of Complex Systems in Material Science, Physics and Biology'' No. 2022HKBF5C. 
	D.R. gratefully acknowledge the Cluster of Excellence EXC 2044-390685587, Mathematics M\"unster: Dynamics-Geometry-Structure funded by the Deutsche
	Forschungsgemeinschaft (DFG, German Research Foundation).
	D.R. and F.S. \ are grateful to Technische Universität München for its kind hospitality during the realization of part of this work, and to the Scuola Superiore Meridionale and the University of Naples "Federico II", respectively, where most of this research was carried out.
	Finally, D.R. and F.S \ are members of the Gruppo Nazionale per l'Analisi Matematica, la Probabilit\'a e le loro Applicazioni (GNAMPA-INdAM).

	\bibliographystyle{siam}
	\bibliography{bibliography_NEW}
	
\end{document}